\documentclass[10pt]{amsart}
\usepackage{amsmath}
\usepackage{amsxtra}
\usepackage{amscd}
\usepackage{amsthm}
\usepackage{amsfonts}
\usepackage{amssymb}
\usepackage{eucal}
\usepackage[all]{xy}
\usepackage{graphicx}
\usepackage[usenames]{color}
\usepackage{tikz-cd}
\usepackage{mathtools}
\usepackage{hyperref}

\newtheorem{cor}[subsubsection]{Corollary}
\newtheorem{lem}[subsubsection]{Lemma}
\newtheorem{prop}[subsubsection]{Proposition}

\newtheorem{thm}[subsubsection]{Theorem}

\newtheorem{defn}[subsubsection]{Definition}

\theoremstyle{remark}

\theoremstyle{definition}

\theoremstyle{remark}

\numberwithin{equation}{section}

\newcommand{\nc}{\newcommand}
\nc{\renc}{\renewcommand}
\nc{\ssec}{\subsection}
\nc{\sssec}{\subsubsection}
\nc{\on}{\operatorname}

\newcommand{\iso}{\buildrel{\sim}\over{\longrightarrow}}

\nc{\ips}{{\iota_P^{(S)}}}
\nc{\ipms}{{\iota_{P^-}^{(S)}}}
\nc{\sfpps}{{\sfp_P^{(S)}}}
\nc{\sfppms}{{\sfp_{P^-}^{(S)}}}

\nc\ol{\overline}
\nc\wt{\widetilde}
\nc\tboxtimes{\wt{\boxtimes}}
\nc\tstar{\wt{\star}}
\nc{\alp}{\alpha}

\nc{\ZZ}{{\mathbb Z}}
\nc{\NN}{{\mathbb N}}
\nc{\OO}{{\mathbb O}}
\renc{\SS}{{\mathbb S}}
\nc{\DD}{{\mathbb D}}
\nc{\GG}{{\mathbb G}}

\nc{\Fq}{{\mathbb F}_q}
\nc{\Fqb}{\ol{{\mathbb F}_q}}
\nc{\Ql}{\ol{{\mathbb Q}_\ell}}
\nc{\id}{\text{id}}
\nc\X{\mathcal X}

\nc{\red}{\on{red}}
\nc{\Ho}{\on{Ho}}
\nc{\Hom}{\on{Hom}}
\nc{\Mor}{\on{Mor}}
\nc{\coef}{\on{coeff}}
\nc{\Lie}{\on{Lie}}
\nc{\Loc}{\on{Loc}}
\nc{\Pic}{\on{Pic}}
\nc{\Bun}{\on{Bun}}
\nc{\IC}{\on{IC}}
\nc{\Aut}{\on{Aut}}
\nc{\rk}{\on{rk}}
\nc{\Sh}{\on{Sh}}
\nc{\Perv}{\on{Perv}}
\nc{\pos}{{\on{pos}}}
\nc{\Conv}{\on{Conv}}
\nc{\Sph}{\on{Sph}}
\nc{\Sym}{\on{Sym}}
\nc{\BunBb}{\overline{\Bun}_B}
\nc{\BunNb}{\overline{\Bun}_N}
\nc{\BunTb}{\overline{\Bun}_T}
\nc{\BunBbm}{\overline{\Bun}_{B^-}}
\nc{\BunBbel}{\overline{\Bun}_{B,el}}
\nc{\BunBbmel}{\overline{\Bun}_{B^-,el}}
\nc{\Buno}{\overset{o}{\Bun}}
\nc{\BunPb}{{\overline{\Bun}_P}}
\nc{\BunBM}{\Bun_{B(M)}}
\nc{\BunBMb}{\overline{\Bun}_{B(M)}}
\nc{\BunPbw}{{\widetilde{\Bun}_P}}
\nc{\BunBP}{\widetilde{\Bun}_{B,P}}
\nc{\GUb}{\overline{G/U}}
\nc{\GUPb}{\overline{G/U(P)}}

\nc{\PP}{\underline{P}'}

\nc{\Hhom}{\underline{\on{Hom}}}
\nc\syminfty{\on{Sym}^{\infty}}
\nc\lal{\ol{\lambda}}
\nc\xl{\ol{x}}
\nc\thl{\ol{\theta}}
\nc\nul{\ol{\nu}}
\nc\mul{\ol{\mu}}
\nc{\Sum}{\Sigma}
\nc{\oX}{\overset{o}{X}{}}
\nc{\hl}{\overset{\leftarrow}h{}}
\nc{\hr}{\overset{\rightarrow}h{}}
\nc{\M}{{\mathcal M}}
\nc{\F}{{\mathcal F}}
\nc{\D}{{\mathcal D}}
\nc{\Q}{{\mathcal Q}}
\nc{\Y}{{\mathcal Y}}
\nc{\G}{{\mathcal G}}
\nc{\E}{{\mathcal E}}
\nc{\CalC}{{\mathcal C}}
\nc\Dh{\widehat{\D}}

\nc{\C}{{\mathcal C}}
\nc{\K}{{\mathcal K}}
\renewcommand{\H}{{\mathcal H}}

\nc{\T}{{\mathcal T}}
\nc{\V}{{\mathcal V}}
\renc{\P}{{\mathcal P}}
\nc{\A}{{\mathcal A}}
\nc{\B}{{\mathcal B}}
\nc{\U}{{\mathcal U}}
\renewcommand{\L}{{\mathcal L}}
\nc{\Gr}{{\on{Gr}}}

\nc{\frn}{{\check{\mathfrak u}(P)}}

\nc{\fC}{\mathfrak C}
\nc{\p}{\mathfrak p}
\nc{\q}{\mathfrak q}
\nc\f{{\mathfrak f}}

\nc{\qo}{{\mathfrak q}}
\nc{\po}{{\mathfrak p}}
\nc{\s}{{\mathfrak s}}
\nc\w{\text{w}}

\nc{\mathi}{\iota}
\nc\Spec{\on{Spec}}
\nc\Proj{\on{Proj}}
\nc\Mod{\on{Mod}}
\nc{\tw}{\widetilde{\mathfrak t}}
\nc{\pw}{\widetilde{\mathfrak p}}
\nc{\qw}{\widetilde{\mathfrak q}}
\nc{\jw}{\widetilde j}

\nc{\grb}{\overline{\Gr}}
\renewcommand{\i}{\mathfrak i}

\nc{\lambdach}{{\check\lambda}}
\nc{\Lambdach}{{\check\Lambda}{}}
\nc{\much}{{\check\mu}}
\nc{\omegach}{{\check\omega}}
\nc{\nuch}{{\check\nu}}
\nc{\etach}{{\check\eta}}
\nc{\alphach}{{\check\alpha}}
\nc{\oblvtach}{{\check\oblvta}}
\nc{\rhoch}{{\check\rho}}
\nc{\ch}{{\check h}}

\nc{\Hb}{\overline{\H}}

\nc{\BA}{{\mathbb{A}}}
\nc{\BC}{{\mathbb{C}}}
\nc{\BG}{{\mathbb{G}}}
\nc{\BM}{{\mathbb{M}}}
\nc{\BO}{{\mathbb{O}}}
\nc{\BD}{{\mathbb{D}}}
\nc{\BBD}{{\mathbf{D}}}
\nc{\BN}{{\mathbb{N}}}
\nc{\BP}{{\mathbb{P}}}
\nc{\BQ}{{\mathbb{Q}}}
\nc{\BR}{{\mathbb{R}}}
\nc{\BZ}{{\mathbb{Z}}}
\nc{\BS}{{\mathbb{S}}}
\nc{\Deep}{{\bf{deep}}}
\nc{\deep}{deep}

\nc{\CA}{{\mathcal{A}}}
\nc{\CB}{{\mathcal{B}}}

\nc{\CE}{{\mathcal{E}}}
\nc{\CF}{{\mathcal{F}}}
\nc{\CH}{{\mathcal{H}}}

\nc{\CL}{{\mathcal{L}}}
\nc{\CC}{{\mathcal{C}}}
\nc{\CG}{{\mathcal{G}}}
\nc{\CalD}{{\mathcal{D}}}
\nc{\CM}{{\mathcal{M}}}
\nc{\CN}{{\mathcal{N}}}
\nc{\CK}{{\mathcal{K}}}
\nc{\CO}{{\mathcal{O}}}
\nc{\CP}{{\mathcal{P}}}
\nc{\CQ}{{\mathcal{Q}}}
\nc{\CR}{{\mathcal{R}}}
\nc{\CS}{{\mathcal{S}}}
\nc{\CT}{{\mathcal{T}}}
\nc{\CU}{{\mathcal{U}}}
\nc{\CV}{{\mathcal{V}}}
\nc{\CW}{{\mathcal{W}}}
\nc{\CX}{{\mathcal{X}}}
\nc{\CY}{{\mathcal{Y}}}
\nc{\CZ}{{\mathcal{Z}}}
\nc{\CI}{{\mathcal{I}}}

\nc{\csM}{{\check{\mathcal A}}{}}
\nc{\oM}{{\overset{\circ}{\mathcal M}}{}}
\nc{\obM}{{\overset{\circ}{\mathbf M}}{}}
\nc{\oCA}{{\overset{\circ}{\mathcal A}}{}}
\nc{\obA}{{\overset{\circ}{\mathbf A}}{}}
\nc{\ooM}{{\overset{\circ}{M}}{}}
\nc{\osM}{{\overset{\circ}{\mathsf M}}{}}
\nc{\vM}{{\overset{\bullet}{\mathcal M}}{}}
\nc{\nM}{{\underset{\bullet}{\mathcal M}}{}}
\nc{\oD}{{\overset{\circ}{\mathcal D}}{}}
\nc{\obD}{{\overset{\circ}{\mathbf D}}{}}
\nc{\oA}{{\overset{\circ}{\mathbb A}}{}}
\nc{\op}{{\overset{\bullet}{\mathbf p}}{}}
\nc{\cp}{{\overset{\circ}{\mathbf p}}{}}
\nc{\oU}{{\overset{\bullet}{\mathcal U}}{}}
\nc{\oZ}{{\overset{\circ}{\mathcal Z}}{}}
\nc{\ofZ}{{\overset{\circ}{\mathfrak Z}}{}}
\nc{\oF}{{\overset{\circ}{\fF}}}

\nc{\fa}{{\mathfrak{a}}}
\nc{\fb}{{\mathfrak{b}}}
\nc{\fd}{{\mathfrak{d}}}
\nc{\ff}{{\mathfrak{f}}}
\nc{\fg}{{\mathfrak{g}}}
\nc{\fgl}{{\mathfrak{gl}}}
\nc{\fh}{{\mathfrak{h}}}
\nc{\fj}{{\mathfrak{j}}}
\nc{\fl}{{\mathfrak{l}}}
\nc{\fm}{{\mathfrak{m}}}
\nc{\fn}{{\mathfrak{n}}}
\nc{\fu}{{\mathfrak{u}}}
\nc{\fp}{{\mathfrak{p}}}
\nc{\fr}{{\mathfrak{r}}}
\nc{\fs}{{\mathfrak{s}}}
\nc{\ft}{{\mathfrak{t}}}
\nc{\fz}{{\mathfrak{z}}}
\nc{\fsl}{{\mathfrak{sl}}}
\nc{\hsl}{{\widehat{\mathfrak{sl}}}}
\nc{\hgl}{{\widehat{\mathfrak{gl}}}}
\nc{\hg}{{\widehat{\mathfrak{g}}}}
\nc{\chg}{{\widehat{\mathfrak{g}}}{}^\vee}
\nc{\hn}{{\widehat{\mathfrak{n}}}}
\nc{\chn}{{\widehat{\mathfrak{n}}}{}^\vee}

\nc{\fA}{{\mathfrak{A}}}
\nc{\fB}{{\mathfrak{B}}}
\nc{\fD}{{\mathfrak{D}}}
\nc{\fE}{{\mathfrak{E}}}
\nc{\fF}{{\mathfrak{F}}}
\nc{\fG}{{\mathfrak{G}}}
\nc{\fK}{{\mathfrak{K}}}
\nc{\fL}{{\mathfrak{L}}}
\nc{\fM}{{\mathfrak{M}}}
\nc{\fN}{{\mathfrak{N}}}
\nc{\fP}{{\mathfrak{P}}}
\nc{\fU}{{\mathfrak{U}}}
\nc{\fV}{{\mathfrak{V}}}
\nc{\fZ}{{\mathfrak{Z}}}

\nc{\bb}{{\mathbf{b}}}
\nc{\bc}{{\mathbf{c}}}
\nc{\bd}{{\mathbf{d}}}
\nc{\bbf}{{\mathbf{f}}}
\nc{\be}{{\mathbf{e}}}
\nc{\bi}{{\mathbf{i}}}
\nc{\bj}{{\mathbf{j}}}
\nc{\bn}{{\mathbf{n}}}
\nc{\bp}{{\mathbf{p}}}
\nc{\bq}{{\mathbf{q}}}
\nc{\bu}{{\mathbf{u}}}
\nc{\bv}{{\mathbf{v}}}
\nc{\bx}{{\mathbf{x}}}
\nc{\bs}{{\mathbf{s}}}
\nc{\by}{{\mathbf{y}}}
\nc{\bw}{{\mathbf{w}}}
\nc{\bA}{{\mathbf{A}}}
\nc{\bK}{{\mathbf{K}}}
\nc{\bB}{{\mathbf{B}}}
\nc{\bC}{{\mathbf{C}}}
\nc{\bG}{{\mathbf{G}}}
\nc{\bD}{{\mathbf{D}}}
\nc{\bH}{{\mathbf{H}}}
\nc{\bM}{{\mathbf{M}}}
\nc{\bN}{{\mathbf{N}}}
\nc{\bV}{{\mathbf{V}}}
\nc{\bW}{{\mathbf{W}}}
\nc{\bX}{{\mathbf{X}}}
\nc{\bZ}{{\mathbf{Z}}}
\nc{\bS}{{\mathbf{S}}}

\nc{\sA}{{\mathsf{A}}}
\nc{\sB}{{\mathsf{B}}}
\nc{\sC}{{\mathsf{C}}}
\nc{\sD}{{\mathsf{D}}}
\nc{\sF}{{\mathsf{F}}}
\nc{\sG}{{\mathsf{G}}}
\nc{\sK}{{\mathsf{K}}}
\nc{\sM}{{\mathsf{M}}}
\nc{\sO}{{\mathsf{O}}}
\nc{\sW}{{\mathsf{W}}}
\nc{\sQ}{{\mathsf{Q}}}
\nc{\sP}{{\mathsf{P}}}
\nc{\sZ}{{\mathsf{Z}}}
\nc{\sfp}{{\mathsf{p}}}
\nc{\bsfp}{{\mathsf{\bar p}_P}}
\nc{\sfq}{{\mathsf{q}}}
\nc{\sr}{{\mathsf{r}}}
\nc{\bk}{{\mathsf{k}}}
\nc{\sg}{{\mathsf{g}}}
\nc{\sff}{{\mathsf{f}}}
\nc{\sfb}{{\mathsf{b}}}
\nc{\sfc}{{\mathsf{c}}}
\nc{\sd}{{\mathsf{d}}}

\nc{\BK}{{\bar{K}}}

\nc{\tA}{{\widetilde{\mathbf{A}}}}
\nc{\tB}{{\widetilde{\mathcal{B}}}}
\nc{\tg}{{\widetilde{\mathfrak{g}}}}
\nc{\tG}{{\widetilde{G}}}
\nc{\TM}{{\widetilde{\mathbb{M}}}{}}
\nc{\tO}{{\widetilde{\mathsf{O}}}{}}
\nc{\tU}{{\widetilde{\mathfrak{U}}}{}}
\nc{\TZ}{{\tilde{Z}}}
\nc{\tx}{{\tilde{x}}}
\nc{\tbv}{{\tilde{\bv}}}
\nc{\tfP}{{\widetilde{\mathfrak{P}}}{}}
\nc{\tz}{{\tilde{\zeta}}}
\nc{\tmu}{{\tilde{\mu}}}

\nc{\urho}{\underline{\rho}}
\nc{\uB}{\underline{B}}
\nc{\uC}{{\underline{\mathbb{C}}}}
\nc{\ui}{\underline{i}}
\nc{\uj}{\underline{j}}
\nc{\ofP}{{\overline{\mathfrak{P}}}}
\nc{\oB}{{\overline{\mathcal{B}}}}
\nc{\og}{{\overline{\mathfrak{g}}}}
\nc{\oI}{{\overline{I}}}

\nc{\eps}{\varepsilon}
\nc{\hrho}{{\hat{\rho}}}

\nc{\one}{{\mathbf{1}}}
\nc{\two}{{\mathbf{t}}}

\nc{\Rep}{{\mathop{\operatorname{\rm Rep}}}}
\nc{\Tot}{{\mathop{\operatorname{\rm Tot}}}}
\nc{\Ker}{{\mathop{\operatorname{\rm Ker}}}}
\nc{\im}{{\mathop{\operatorname{\rm Im}}}}
\nc{\Hilb}{{\mathop{\operatorname{\rm Hilb}}}}
\nc{\End}{{\mathop{\operatorname{\rm End}}}}
\nc{\Ext}{{\mathop{\operatorname{\rm Ext}}}}
\nc{\CHom}{{\mathop{\operatorname{{\mathcal{H}}\it om}}}}
\nc{\GL}{{\mathop{\operatorname{\rm GL}}}}
\nc{\gr}{{\mathop{\operatorname{\rm gr}}}}
\nc{\HN}{{\mathop{\operatorname{\rm HN}}}}
\nc{\Id}{{\mathop{\operatorname{\rm Id}}}}
\nc{\de}{{\mathop{\operatorname{\rm def}}}}
\nc{\length}{{\mathop{\operatorname{\rm length}}}}
\nc{\supp}{{\mathop{\operatorname{\rm supp}}}}

\nc{\Cliff}{{\mathsf{Cliff}}}
\nc{\Fl}{\on{Fl}}
\nc{\Fib}{{\mathsf{Fib}}}
\nc{\Coh}{{\on{Coh}}}
\nc{\QCoh}{{\on{QCoh}}}
\nc{\IndCoh}{{\on{IndCoh}}}
\nc{\FCoh}{{\mathsf{FCoh}}}

\nc{\reg}{{\text{\rm reg}}}

\nc{\cplus}{{\mathbf{C}_+}}
\nc{\cminus}{{\mathbf{C}_-}}
\nc{\cthree}{{\mathbf{C}_*}}
\nc{\Qbar}{{\bar{Q}}}
\nc\Eis{\on{Eis}}
\nc\Eisb{\ol\Eis{}}
\nc\Eisr{\on{Eis}^{rat}{}}
\nc\wh{\widehat}
\nc{\Def}{\on{Def_{\check{\fb}}(E)}}
\nc{\barZ}{\overline{Z}{}}
\nc{\barbarZ}{\overline{\barZ}{}}
\nc{\barpi}{\overline\pi}
\nc{\barbarpi}{\overline\barpi}
\nc{\barpip}{\overline\pi{}^+}
\nc{\barpim}{\overline\pi{}^-}

\nc{\fq}{\mathfrak q}

\nc{\fqb}{\ol{\sfq}{}}
\nc{\fpb}{\ol{\sfp}{}}
\nc{\fpr}{{\sfp^{rat}}{}}
\nc{\fqr}{{\sfq^{rat}}{}}

\nc{\hattimes}{\wh\otimes}

\nc{\bh}{{\bar{h}}}
\nc{\bOmega}{{\overline{\Omega(\check \fn)}}}

\nc{\seq}[1]{\stackrel{#1}{\sim}}

\nc{\cT}{{\check{T}}}
\nc{\cG}{{\check{G}}}
\nc{\cM}{{\check{M}}}
\nc{\cB}{{\check{B}}}

\nc{\ct}{{\check{\mathfrak t}}}
\nc{\cg}{{\check{\fg}}}
\nc{\cb}{{\check{\fb}}}
\nc{\cn}{{\check{\fn}}}

\nc{\cLambda}{{\check\Lambda}}

\nc{\cla}{{\check\lambda}}
\nc{\cmu}{{\check\mu}}
\nc{\cnu}{{\check\nu}}
\nc{\ceta}{{\check\eta}}

\nc{\DefbE}{{\on{Def}_{\cB}(E_\cT)}}

\nc{\imathb}{{\ol{\imath}}}
\nc{\rlr}{\overset{\longrightarrow}{\underset{\longrightarrow}\longleftarrow}}

\nc{\oBun}{\overset{\circ}\Bun}
\nc{\LocSys}{\on{LocSys}}
\nc{\BunBbb}{\ol{\ol{Bun}}_B}
\nc{\BunBr}{\Bun_B^{rat}}
\nc{\BunBrsg}{\Bun_B^{rat,\on{s.g.}}}
\nc{\BunBrp}{\Bun_B^{rat,polar}}
\nc{\BunBrpbg}{\Bun_B^{rat,polar,\on{b.g.}}}
\nc{\BunBrpsg}{\Bun_B^{rat,polar,\on{s.g.}}}
\nc{\BunTrp}{\Bun_T^{rat,polar}}
\nc{\BunTrpbg}{\Bun_T^{rat,polar,\on{b.g.}}}
\nc{\BunTrpsg}{\Bun_T^{rat,polar,\on{s.g.}}}
\nc{\BunNr}{\Bun_N^{rat}}
\nc{\BunNre}{\Bun_N^{enh,rat}}
\nc{\BunTr}{\Bun_T^{rat}}
\nc{\Vect}{\on{Vect}}
\nc{\Whit}{\on{Whit}}
\nc{\CTb}{\ol{\on{CT}}}
\nc{\Ran}{\on{Ran}}
\nc{\CTr}{\on{CT}^{rat}{}}
\nc\jmathr{\jmath^{rat}{}}
\nc{\ux}{\underline{x}}
\nc{\clambda}{{\check\lambda}}
\nc{\calpha}{{\check\alpha}}
\nc{\ind}{{\mathbf{ind}}}
\nc{\oblv}{{\mathbf{oblv}}}
\nc{\ox}{{\overline{x}}}
\nc{\cLa}{\check{\Lambda}}
\nc{\StinftyCat}{\on{DGCat}}
\nc{\inftyCat}{\infty\on{-Cat}}
\nc{\inftygroup}{\infty\on{-Grpd}}
\nc{\Dmod}{\on{D-mod}}
\nc{\CMaps}{{\mathcal Maps}}
\nc{\Maps}{\on{Maps}}
\nc{\affSch}{\on{Sch}^{\on{aff}}}
\nc{\dr}{{\on{dR}}}
\nc{\rD}{{\blacktriangle}}
\nc{\oCY}{\overset{\circ}\CY}
\nc{\leqG}{\underset{G}\leq}
\nc{\leqM}{\underset{M}\leq}
\nc{\leqGad}{\underset{G_{ad}}\leq}
\nc{\leqMad}{\underset{M_{ad}}\leq}
\nc{\psId}{\on{Ps-Id}}

\nc{\sotimes}{\overset{!}\otimes}

\begin{document}

\title[$\Dmod(\Bun_G^{\on{I}})$ is Compactly Generated]{$\Dmod(\Bun_G^{\on{I}})$ is Compactly Generated}

\author{Taeuk Nam}

\date{\today}

\begin{abstract}
Drinfeld and Gaitsgory proved that $\Dmod(\Bun_G)$ is compactly generated in \cite{DG2}. Let $\Bun_G^{\on{I}}$ be the algebraic stack of principal $G$-bundles on $X$ together with Iwahori level structure at a fixed point $x \in X$. More generally, for a finite collection of points $x_1, ..., x_k \in X$, let $\Bun_G^{(\on{I}; x_1, ..., x_k)}$ be the algebraic stack of principal $G$-bundles on $X$ together with Iwahori level structure at each point $x_j$. We will show that $\Dmod(\Bun_G^{\on{I}})$ and $\Dmod(\Bun_G^{(\on{I}; x_1, ..., x_k)})$ are compactly generated.
\end{abstract}

\maketitle

\tableofcontents

\section*{Introduction}

\ssec{Motivation}

\sssec{}
In \cite{DG2}, Drinfeld and Gaitsgory showed that the DG category $\Dmod(\Bun_G)$ enjoys a certain smallness property called \emph{compact generation}. Roughly speaking, the \emph{compact objects} of a category $\bC$ are those objects $c \in \bC$ such that mapping out of $c$ is well-behaved (i.e. commutes with colimits), and the category $\bC$ is said to be \emph{compactly generated} if every object of $\bC$ can be written as a colimit of compact objects.

\sssec{}
Being compactly generated is a highly desirable finiteness condition on a DG category for the following two reasons, one internal and one external.

\sssec{}
The first reason is that if $\bC$ is compactly generated, then it suffices to understand maps between compact objects to understand maps between arbitrary objects. More concretely, if $c, d \in \bC$, we can write
$$
c = \underset{i \in I}{\on{colim}} \ c_i, d = \underset{j \in J}{\on{colim}} \ d_j
$$
where the $c_i$ and $d_j$ are compact objects. Then we have
\begin{align}
\CMaps_{\bC}(c, d) &= \CMaps_{\bC}(\underset{i \in I}{\on{colim}} \ c_i, \underset{j \in J}{\on{colim}} \ d_j) \\
&= \underset{i \in I}{\on{lim}} \ \CMaps_{\bC}(c_i, \underset{j \in J}{\on{colim}} \ d_j) \\
&= \underset{i \in I}{\on{lim}} \ \underset{j \in J}{\on{colim}} \ \CMaps_{\bC}(c_i, d_j).
\end{align}

\sssec{}
The second reason is that if $\bC$ is compactly generated, then (continuous) functors out of $\bC$ are determined by where they send the compact objects. In other words, if $\bD$ is another (not necessarily compactly generated) cocomplete DG category, then we have an identification
$$
\on{Fun}_{cts}(\bC, \bD) = \on{Fun}_{exact}(\bC^c, \bD)
$$
where $\bC^c$ is the full subcategory of compact objects.

\sssec{}
The following is one source of motivation for wanting to study $\Dmod(\Bun_G)$ in particular.

\sssec{}
The (de Rham, global, critical level) unramified geometric Langlands conjecture states that there is an equivalence of DG categories
$$
\Dmod(\Bun_G) \xrightarrow{\sim} \IndCoh_{\on{Nilp}}(\LocSys_{G^{\vee}})
$$
that satisfies a number of conditions.

\sssec{}
The right hand side of the equivalence, $\IndCoh_{\on{Nilp}}(\LocSys_{G^{\vee}})$, is equivalent to the ind-completion $\on{Ind}(\Coh_{\on{Nilp}}(\LocSys_{G^{\vee}}))$, so it is \emph{a priori} compactly generated. Thus, the compact generation of $\Dmod(\Bun_G)$ is a necessary condition for the geometric Langlands conjecture to hold.

\sssec{}
The conjecture was recently proven in a series of five papers \cite{GLC1}, \cite{GLC2}, \cite{GLC3}, \cite{GLC4}, \cite{GLC5} by nine authors: Arinkin, Beraldo, Campbell, Chen, Faergeman, Gaitsgory, Lin, Raskin, and Rozenblyum.

\sssec{}
The first step of the proof, contained in \cite{GLC1}, is the construction of the Langlands functor
$$
\mathbb{L}_G : \Dmod(\Bun_G) \to \IndCoh_{\on{Nilp}}(\LocSys_{G^{\vee}}).
$$
The point at which the compact generation of $\Dmod(\Bun_G)$ is used is as follows: there is a monoidal action
$$
\QCoh(\LocSys_{G^{\vee}}) \otimes \Dmod(\Bun_G) \to \Dmod(\Bun_G)
$$
called the \emph{spectral action}. By fixing a certain object
$$
\on{Poinc}_{G, !}^{\on{Vac, glob}} \in \Dmod(\Bun_G)
$$
called the \emph{vacuum Poincare sheaf}, we obtain a functor
$$
\QCoh(\LocSys_{G^{\vee}}) \to \Dmod(\Bun_G).
$$
This functor has a continuous right adjoint, which we denote
$$
\mathbb{L}_{G, \on{coarse}} : \Dmod(\Bun_G) \to \QCoh(\LocSys_{G^{\vee}}).
$$
Then it is shown in \cite[Theorem 1.6.2]{GLC1} that $\mathbb{L}_{G, \on{coarse}}$ sends compact objects in $\Dmod(\Bun_G)$ to eventually coconnective objects in $\QCoh(\LocSys_{G^{\vee}})$. In other words, we have an exact functor
$$
\Dmod(\Bun_G)^c \to \QCoh(\LocSys_{G^{\vee}})^{>-\infty}
$$
and since the functor
$$
\Psi_{\on{Nilp}, \{0\}} : \IndCoh_{\on{Nilp}}(\LocSys_{G^{\vee}}) \to \QCoh(\LocSys_{G^{\vee}})
$$
is an equivalence on the eventually coconnective full subcategories, this yields
$$
\Dmod(\Bun_G)^c \to \IndCoh_{\on{Nilp}}(\LocSys_{G^{\vee}})^{>-\infty} \subset \IndCoh_{\on{Nilp}}(\LocSys_{G^{\vee}})
$$
and finally, because $\Dmod(\Bun_G)$ is compactly generated, we obtain a continuous functor
$$
\mathbb{L}_G : \Dmod(\Bun_G) \to \IndCoh_{\on{Nilp}}(\LocSys_{G^{\vee}}).
$$

\sssec{}
A natural direction of further inquiry and generalization is to go from the unramified setting to the ramified setting. This means that instead of considering the stack $\Bun_G$, we instead consider the moduli stack of principal $G$-bundles on $X$ with \emph{level structure} at a point $x \in X$.

\sssec{}
Let $D_x$ be the formal disc around $x \in X$. Restricting a principal $G$-bundle $\F_G$ along $D_x \to X$ yields a principal $G$-bundle $\F_{G, D_x}$ on $D_x$. Then we can interpret $G(\CO)$ as the group of sections of the constant group scheme $D_x \times G$ over $D_x$ and consider $\F_{G, D_x}$ as a $G(\CO)$-torsor. For a subgroup $H \subset G(\CO)$, $H$-level structure (at $x$) on $\F_G$ is the data of a reduction of $\F_{G, D_x}$ to $H$. The moduli stack of principal $G$-bundles on $X$ with $H$-level structure at $x \in X$ is denoted by $\Bun_G^{H}$.

\sssec{}
A side remark is that we could of course consider $H$-level structure at any arbitrary finite collection of points $x_1, ..., x_k \in X$ (or even level structure for a different subgroup at each point). However, for the sake of ease of exposition, here and in the sequel, we will discuss the single point case and only bring up the multiple point case as necessary.

\sssec{}
Let us take a moment to note that in the case of $H = G(\CO)$, $H$-level structure is no additional structure at all, i.e. $\Bun_G^{G(\CO)} = \Bun_G$. The \emph{smaller} the subgroup $H$, the \emph{wilder} the ramification; for example, the case of $H = \{e\}$ is called \emph{full level structure}, and the corresponding stack is denoted $\Bun_G^{\infty x}$.

\sssec{}
This leads to the natural question: for any subgroup $H \subset G(\CO)$, can we describe the category $\Dmod(\Bun_G^{H})$ in spectral terms? i.e. is there a ramified geometric Langlands equivalence?

\sssec{}
For an arbitrary subgroup $H \subset G(\CO)$, this question is out of reach with current technology. Thus as a first step, it makes sense to consider large subgroups of $G(\CO)$ (or in other words, \emph{tame} ramification). More precisely, there is a (surjective) group homomorphism
$$
G(\CO) \to G
$$
by evaluation at $t = 0$, and we call its kernel the \emph{first congruence subgroup} $G^{(1)}(\CO)$. We say a subgroup $H \subset G(\CO)$ is \emph{tame} if it contains $G^{(1)}(\CO)$.

\sssec{}
The geometric significance of tameness is that for a tame subgroup $H$, $H$-level structure is the same as the data of a $H/G^{(1)}(\CO)$-reduction at $x$. In plain language, tame level structure lives \emph{only} at the point itself, \emph{not} on the formal neighborhood.

\sssec{}
We will focus on the tame subgroup $\on{I}$, called the \emph{Iwahori} subgroup. This is the preimage of the borel $B \subset G$ under the homomorphism $G(\CO) \to G$. Thus Iwahori level structure is the data of a $B$-reduction at $x \in X$.

\sssec{}
As a step towards the Iwahori ramified geometric Langlands conjecture, in the body of this paper we will show that $\Dmod(\Bun_G^{\on{I}})$ is compactly generated.

\ssec{Main Result}

\sssec{}
In \cite{DG2}, Drinfeld and Gaitsgory proved that the DG category of D-modules on the algebraic stack $\Bun_G$ is compactly generated.

\sssec{}
This paper builds upon \cite{DG2} in the following way: instead of $\Bun_G$, we consider the algebraic stack $\Bun_G^{\on{I}}$ classifying principal $G$-bundles on $X$ together with Iwahori level structure at a fixed point $x \in X$ (i.e. a $B$-reduction at a fixed point $x \in X$). Our main result states that

\begin{thm} \label{main result 1}
    $\Dmod(\Bun_G^{\on{I}})$ is compactly generated.
\end{thm}

\sssec{}
More generally, for a finite collection of points $x_1, ..., x_k \in X$, we can consider the algebraic stack $\Bun_G^{(\on{I}; x_1, ..., x_k)}$ classifying principal $G$-bundles on $X$ together with Iwahori level structure at each point $x_j$. Our argument will go through essentially verbatim in this generality, showing that

\begin{thm} \label{main result 2}
    $\Dmod(\Bun_G^{(\on{I}; x_1, ..., x_k)})$ is compactly generated.
\end{thm}

\sssec{}
It has been pointed out to us that an unpublished article \cite{AG} by Gaitsgory and Arkhipov sketches the proof of compact generation in both the unramified and Iwahori ramified case. However, this argument has a deficiency of a combinatorial nature which was fixed in \cite{DG2} for the unramified case. Our result shows that a similar modification is sufficient for the Iwahori ramified case.

\ssec{The Structure of the Paper}

\sssec{}
Here is a description of how the paper is organized.

\sssec{}
In Section \ref{BunG}, we begin by introducing the definition of compactly generated DG category in Subsection \ref{DGCat}. Then we summarize Drinfeld and Gaitsgory's proof in \cite{DG2} that $\Dmod(\Bun_G)$ is compactly generated. Here is a very brief overview of the argument:

\sssec{}
First, a prior result in \cite{DG1}, by the same authors, says that QCA algebraic stacks, i.e. those that are quasicompact ("QC") and have affine automorphism groups ("A"), have compactly generated DG categories of D-modules. We give a summary of the discussion in Subsection \ref{qca stacks}.

\sssec{}
The complication is that $\Bun_G$ is not QCA because it fails to be quasicompact. However, $\Bun_G$ does have affine automorphism groups, so it is \emph{locally} QCA. This means that $\Bun_G$ is a union of quasicompact QCA open substacks, so $\Dmod(\Bun_G)$ is a limit of compactly generated DG categories, under $!$-pullback.

\sssec{}
If these $!$-pullbacks had well defined left adjoints, this limit would be equivalent to taking the colimit where the connecting morphisms are the left adjoints of the $!$-pullbacks, the $!$-pushforward functors. Then since these $!$-pushforward functors would have continuous right adjoints, the colimit would automatically be compactly generated for formal reasons. Unfortunately, $!$-pushforward functors are not well defined along all maps.

\sssec{}
A property of open substacks called \emph{co-truncativeness} on $U$ guarantees all $!$-pushforward functors along inclusions $U \to U'$ are well-defined. This is discussed in Subsection \ref{t and t}.

\sssec{}
We then review in Subsection \ref{contractive substacks} a geometric condition on the closed complement of $U$ called \emph{contractiveness} that is sufficient for $U$ to be co-truncative.

\sssec{}
The explicit stratification of $\Bun_G$ into quasicompact open substacks given in \cite{DG2}, roughly speaking, is called the Harder-Narasimhan stratification and has to do with the degrees of the reductions of the principal $G$-bundles to parabolic subgroups. We describe this stratification in Subsection \ref{stratification}.

\sssec{}
Finally, we sketch the proof that the closed complements of the quasicompact open substacks above are indeed contractive (and therefore that the quasicompact open substacks are indeed co-truncative) in Subsection \ref{truncatability}.

\sssec{}
In Section \ref{SL2}, we go through the proof of compact generation, in both the unramified setting (in Subsection \ref{SL2 unramified}) and the Iwahori ramified setting (in Subsection \ref{SL2 Iwahori}), for $G = \on{SL}_2$. Most of the conceptual ideas of the proof appear already in this toy example, so a reader wanting a broad understanding of how the proof works without thinking too carefully about the combinatorics on first pass can skip Subsections \ref{stratification} and \ref{truncatability} to proceed directly to Section \ref{SL2}.

\sssec{}
In Section \ref{BunGI}, we prove the main result of the paper, Theorem \ref{main result 1}, that $\Dmod(\Bun_G^{\on{I}})$ is compactly generated.

\sssec{}
We spend Subsections \ref{reduction steps} and \ref{outline} giving an analogous stratification of $\Bun_G^{\on{I}}$. Very roughly speaking, we stratify first by the strata of the underlying principal $G$-bundle in the Harder-Narasimhan stratification, then stratify further by the relative position of the Iwahori level structure (i.e. the $B$-reduction at a point) and the $P$-reduction.

\sssec{}
We then spend Subsections \ref{global generation subsection} to \ref{contractive} showing that the closed complements of the quasicompact open substacks of our stratification are actually contractive.

\sssec{}
Finally, we discuss the proof of Theorem \ref{main result 2}, the compact generation of $\Dmod(\Bun_G^{(\on{I}, x_1, ..., x_k)})$, in Subsection \ref{multiple points}, explaining that our entire argument goes through modulo only a single change relating to a certain numerical bound in the definition of the stratification.

\ssec{Acknowledgements}
The author would like to express their deepest gratitude towards their advisor Dennis Gaitsgory, for his invaluable mentorship and support throughout the course of this project.

The author is also grateful towards Wyatt Reeves for suggesting the problem that would become our main result.

In addition, the author extends sincere thanks to Kevin Lin, for his careful reading of a previous draft of this paper and his thoughtful comments which greatly improved the clarity and quality of this manuscript.

Finally, the author is indebted to all three of the above individuals, as well as Elden Elmanto, Sanath Devalapurkar, Ekaterina Bogdanova, Grant Barkley, Johnny Gao, and countless other people for the numerous helpful conversations and patient explanations without which this paper would not exist.

Part of this project was done during the author's time at the Max Planck Institute for Mathematics in Bonn.

\section{Compact Generation of $\Dmod(\Bun_G)$} \label{BunG}

In this section, we give a brief summary of what has already been proven by Drinfeld and Gaitsgory in \cite{DG2}, i.e. the compact generation of $\Dmod(\Bun_G)$. We emphasize that none of the results in this section are due to us.

\ssec{Compact Generation of DG Categories} \label{DGCat}

\sssec{}
Our setting is $\StinftyCat_{\on{cont}}$; in other words our DG categories are assumed to be cocomplete and the functors between them are assumed to be continuous unless otherwise noted.

\sssec{}
An object $\bc \in \bC$ is said to be \emph{compact} if 
$$
 \CMaps_{\bC}(\bc, -) : \bC \to \Vect
$$
is a continuous functor.

\begin{defn} \label{def of cpct gen} \cite[1.2.2]{DG2}
    Let $\bC$ be a DG category. We say that $\bC$ is \emph{compactly generated} if there exists a collection of compact objects $\{\bc_\alpha \in \bC \}_{\alpha \in I}$ such that 
    $$
    \CMaps_{\bC}(\bc_\alpha, \bc') = 0 \ \forall \alpha
    $$
    implies $\bc' = 0$.
\end{defn}

\sssec{}
The next two propositions concern stability properties of being compactly generated.

\begin{prop} \label{cpct gen functor}
    Let $\bC$ and $\bD$ be a pair of DG categories, and let
    $$\sF : \bC \to \bD$$
    be a (continuous) functor between them. Then $\sF$ has a (not necessarily continuous) right adjoint 
    $$\sG : \bD \to \bC.$$
    Suppose $\sG$ is continuous and conservative. If $\bC$ is compactly generated, then so is $\bD$.
\end{prop}

\sssec{}
Now consider a functor 
$$
 \Psi : I \to \StinftyCat_{\on{cont}}
$$
that sends
$$
 (i \in I) \mapsto \bC_i, \ (i \to j \in I) \mapsto (\psi_{i,j}: \bC_i \to \bC_j).
$$
For each $i \to j \in I$, 
$$
 \psi_{i,j}: \bC_i \to \bC_j
$$
has a right adjoint 
$$\phi_{j,i} : \bC_j \to \bC_i.$$
Suppose $\phi_{j,i}$ is continuous for all $i \to j$. Then we obtain a functor
$$
 \Phi : I^{\on{op}} \to \StinftyCat_{\on{cont}}
$$
and it is a fact (see \cite[Proposition 1.7.5]{DG2}) that
$$
 \underset{i \in I}{\underset{\longrightarrow}{colim}} \ (\bC_i, \psi_{i, j})
$$
is canonically equivalent to 
$$
 \underset{i \in I^{\on{op}}}{\underset{\longleftarrow}{lim}} \ (\bC_i, \phi_{j, i}).
$$

\begin{prop} \label{cpct gen colim} \cite[Corollary 1.9.4]{DG2}
    In the above situation, if $\bC_i$ is compactly generated for all $i \in I$, then $\underset{i \in I}{\underset{\longrightarrow}{colim}} \  (\bC_i, \psi_{i,j})$ is compactly generated.
\end{prop}

\ssec{QCA Stacks} \label{qca stacks}

\sssec{}
In this subsection, we will give an extremely brief summary of the parts of \cite{DG1} that pertains to the compact generation of $\Dmod$.

\sssec{}
First, let us recall the definition of $\IndCoh$. Given a quasicompact (locally noetherian and almost of finite type) DG scheme $Z$, we can consider
$$
  \Coh(Z) \subset \QCoh(Z)
$$
and define
$$
  \IndCoh(Z) := \on{Ind} (\Coh(Z)).
$$

\sssec{}
Given a map $f : Z_1 \to Z_2$, there is a (continuous) functor
$$
  f^! : \IndCoh(Z_2) \to \IndCoh(Z_1)
$$
defined in \cite{IndCoh}. This data upgrades to a functor
$$
  \IndCoh^! : (\on{DGSch}_{\on{aft}})^{\on{op}} \to \StinftyCat_{\on{cont}}
$$
which sends 
$$(Z \in \on{DGSch}_{\on{aft}}) \mapsto \IndCoh(Z), \ (f : Z_1 \to Z_2) \mapsto (f^! : \IndCoh(Z_2) \to \IndCoh(Z_1)).$$

\sssec{}
Let $\on{PreStk}$ denote the category of (locally almost of finite type) prestacks. We can left Kan extend $\IndCoh^!$ along the inclusion
$$
  \on{DGSch}_{\on{aft}}^{\on{op}} \to \on{PreStk}^{\on{op}} 
$$
to obtain a functor, also denoted
$$
  \IndCoh^! : \on{PreStk}^{\on{op}} \to \StinftyCat_{\on{cont}}
$$
by a slight abuse of notation. More concretely, we have
$$
  \IndCoh(\CY) = \underset{Z \to \CY}{\underset{\longleftarrow}{lim}} \IndCoh(Z)
$$
where the limit is taken over quasicompact DG schemes $Z$ mapping to $\CY$. In other words, the data of an ind-coherent sheaf $\CF$ on a prestack $\CY$ is equivalent to the data of a compatible family of ind-coherent sheaves $g^! \CF \in \IndCoh(Z)$ for each map $g : Z \to \CY$ from a quasicompact DG scheme.

\sssec{}
Now let $\CY$ be an algebraic stack. We can consider the full subcategory
$$
  \Coh(\CY) \subset \IndCoh(\CY)
$$
consisting of $\CF \in \IndCoh(\CY)$ such that $g^! \CF \in \Coh(Z)$ for all smooth maps $g: Z \to \CY$ from quasicompact DG schemes.

\sssec{}
The inclusion
$$
  \Coh(\CY) \subset \IndCoh(\CY)
$$
yields a functor
$$
  \on{Ind} (\Coh(\CY)) \to \IndCoh(\CY) 
$$
and if $\CY$ is a quasicompact $DG$ scheme, by definition of $\IndCoh$ this functor is an equivalence. However, in general it need not be.

\begin{defn} \cite[Definition 2.2.2]{DG2} 
    Let $\CY$ be an algebraic stack. $\CY$ is \emph{locally QCA} if the automorphism groups of geometric points of $\CY$ are affine. $\CY$ is \emph{QCA} if it is quasicompact and locally QCA.
\end{defn}

\begin{prop} \label{Ind of Coh is IndCoh} \cite[Theorem 3.3.5]{DG1}
    Let $\CY$ be a QCA algebraic stack. Then the functor
    $$
    \on{Ind} (\Coh(\CY)) \to \IndCoh(\CY)
    $$
    defined above is an equivalence. In particular, for a QCA algebraic stack $\CY$, $\IndCoh(\CY)$ is compactly generated by $\Coh(\CY)$.
\end{prop}

\sssec{}
Now we will discuss the relationship between $\IndCoh$ and $\Dmod$. Recall that by definition, for a (locally almost of finite type) prestack $\CY$,
$$
\Dmod(\CY) := \IndCoh(\CY_{\on{dR}}).
$$
The map $\iota : \CY \to \CY_{dR}$ gives rise to a continuous and conservative functor
$$
 \iota^! : \Dmod(\CY) = \IndCoh(\CY_{\on{dR}}) \to \IndCoh(\CY)
$$
which we will denote by $\oblv_{\Dmod(\CY)}$.

\sssec{}
For a general prestack, $\oblv_{\Dmod(\CY)}$ need not admit a left adjoint. However, when $\CY$ is an algebraic stack, $\oblv_{\Dmod(\CY)}$ admits a left adjoint
$$\ind_{\Dmod(\CY)} : \IndCoh(\CY) \to \Dmod(\CY)$$ (see \cite[6.3]{DG1}). Proposition \ref{cpct gen functor} and Proposition \ref{Ind of Coh is IndCoh} together imply 

\begin{prop} \label{Dmod of QCA}
    Let $\CY$ be a QCA algebraic stack. Then $\Dmod(\CY)$ is compactly generated.
\end{prop}

\ssec{Truncativeness and Truncatability} \label{t and t}

\sssec{}
Recall that the goal of this section is to summarize the argument that $\Dmod(\Bun_G)$ is compactly generated. If $\Bun_G$ were QCA, we would be done by Proposition \ref{Dmod of QCA}. Unfortunately, $\Bun_G$ is not QCA because it is not quasicompact. However, it is \emph{locally} QCA, i.e. its field valued points have affine automorphism groups.

\sssec{} \label{need for cotruncativeness}
Let $\CY$ be a locally QCA algebraic stack. Denote by $\on{Open}(\CY)$ the poset of quasicompact open substacks $U \subset \CY$. We have
$$
  \Dmod(\CY) = \underset{U_i \in \on{Open}(\CY)^{\on{op}}}{\underset{\longleftarrow}{lim}} \Dmod(U_i)
$$
where for an inclusion $j_{i,j} : U_i \hookrightarrow U_j$ the connecting functor is 
$$
j_{i,j}^! : \Dmod(U_j) \to \Dmod(U_i).
$$
If the functors $j^!$ admitted well-defined left adjoints, we would be in the situation of Proposition \ref{cpct gen colim} and we could conclude that 
$$
\Dmod(\CY) = \underset{U_i \in \on{Open}(\CY)^{\on{op}}}{\underset{\longleftarrow}{lim}} \Dmod(U_i) = \underset{U_i \in \on{Open}(\CY)}{\underset{\longrightarrow}{colim}} \Dmod(U_i)
$$
is compactly generated. Unfortunately, $j_{i,j}^!$ does not always admit a left adjoint.

\begin{defn} \cite[Definition 3.1.5]{DG2}
    Let $\CY$ be a QCA algebraic stack. Let $i : \CZ \hookrightarrow \CY$ be a closed substack and let $j : U \hookrightarrow \CY$ be its complementary open. We say that $\CZ \hookrightarrow \CY$ is \emph{truncative} and that $U \hookrightarrow \CY$ is \emph{co-truncative} if $i_*$ admits a well-defined left adjoint
    $$
    i^* : \Dmod(\CY) \to \Dmod(\CZ)
    $$
    or equivalently, if $j^!$ admits a well-defined left adjoint
    $$
      j_! : \Dmod(U) \to \Dmod(\CY).
    $$
\end{defn}

\sssec{}
For a proof that the two conditions in the above definition, as well as a couple of other conditions, are equivalent, see \cite[Proposition 3.1.2]{DG2}.

\begin{defn} \cite[Definition 3.8.2]{DG2}
    Now let $\CY$ be a locally QCA algebraic stack. We say that a closed substack $i : \CZ \hookrightarrow \CY$ is \emph{truncative} if for every quasicompact open substack $\oCY \subset \CY$,
    $$
      \CZ \cap \oCY \hookrightarrow \oCY
    $$
    is truncative. Similarly, we say that an open substack $j : U \hookrightarrow \CY$ is \emph{co-truncative} if for every quasicompact open substack $\oCY \subset \CY$,
    $$
      U \cap \oCY \hookrightarrow \oCY
    $$
    is co-truncative.
\end{defn}

\sssec{}
Suppose that $\CZ \hookrightarrow \CY$ is now only locally closed. Then we can factor the inclusion as
$$
  \CZ \hookrightarrow \CV \hookrightarrow \CY
$$
where $\CZ \hookrightarrow \CV$ is a closed embedding and $\CV \hookrightarrow \CY$ is an open embedding. Then we say that $\CZ \hookrightarrow \CY$ is truncative if $\CZ \hookrightarrow \CV$ is.

\sssec{}
One can formally deduce that even in the case where $\CY$ is only locally QCA, $j : U \hookrightarrow \CY$ being cotruncative implies that $j^!$ has a well-defined left adjoint, which we still denote by $j_!$.

\begin{defn}\cite[Definition 4.1.1]{DG2}
    A locally QCA algebraic stack $\CY$ is \emph{truncatable} if it can be covered by co-truncative quasicompact open substacks. Alternatively, $\CY$ is truncatable if $\on{CTrnk}(\CY)$, the subposet of $\on{Open}(\CY)$ consisting of the co-truncative quasicompact open substacks, is cofinal in $\on{Open}(\CY)$.
\end{defn}

\sssec{}
The discussion in subsubsection \ref{need for cotruncativeness} implies that

\begin{prop} \label{truncatable}
    Let $\CY$ be a truncatable locally QCA algebraic stack. Then $\Dmod(\CY)$ is compactly generated.
\end{prop}

\sssec{}

One might object that for $U \in \on{CTrnk}(\CY)$, we only have that $j_!$ is defined for $j : U \hookrightarrow \CY$, whereas in the discussion in subsubsection \ref{need for cotruncativeness}, we needed $j_{i, j !}$ to be defined for $j_{i,j} : U_i \hookrightarrow U_j$. However, note that we have a pullback square
$$
\begin{tikzcd}
    U_1 \arrow[r, hook, "j_{12}"] \arrow[d, equal] \arrow[dr, phantom, "\lrcorner", very near start] & U_2 \arrow[d, hook, "j_2"] \\
    U_1 \arrow[r, hook, "j_1"'] & \CY
\end{tikzcd}
$$
and by base change for $\Dmod$, we have that $j_{1,2}^! = j_{1}^! j_{2 *}$, so
$$
  j_{1,2 !} = j_2^* j_{1 !}.
$$

\ssec{Contractive Substacks} \label{contractive substacks}

\sssec{}
In this subsection, we give a sufficient condition of a geometric nature for a substack to be truncative. Consider, as in \cite[5.1.1]{DG2}, the following setup. 

\begin{defn}
Consider $\BA^1$ as a monoid under multiplication, and let $W$ be an $\BA^1$-scheme. Let $p : W \to S$ be an $\BA^1$-equivariant affine morphism of schemes, where the action of $\BA^1$ on $S$ is trivial. Let $\iota : S \to W$ be a section of $p$ such that the composition
$$
  W \to S \to W
$$
corresponds to the action of $0 \in \BA^1$. Let us denote $\CY := W/\BG_m$ and $\CZ := S/\BG_m = S \times \on{pt}/\BG_m$. We have maps
$$
  i : \CZ \hookrightarrow \CY
$$
and
$$
  \pi : \CY \to \CZ
$$
obtained from $\iota$ and $p$ respectively. We will call a closed substack of the form of $i : \CZ \hookrightarrow \CY$ \emph{globally contractive}.
\end{defn}

\sssec{}
The contraction principle (see \cite[Proposition 5.3.2]{DG2}) says that in the above situation, we have an adjunction
$$
  \pi_* : \Dmod(W/\BG_m) \rightleftarrows \Dmod(S/\BG_m) : i_*
$$
so we can identify $i^*$ with $\pi_*$, which is well-defined. In other words, a globally contractive substack is automatically truncative. We can generalize this notion slightly and define

\begin{defn} \cite[5.2.1]{DG2}
    A locally closed substack $\CZ' \hookrightarrow \CY'$ is called \emph{contractive} if there exists a commutative diagram
    $$
    \begin{tikzcd}
        \CZ \arrow[r] \arrow[d]
        & \CY \arrow[d] \\
        \CZ' \arrow[r]
        & \CY'
    \end{tikzcd}
    $$
    such that 
    \begin{enumerate}
        \item[(i)] $\CZ \to \CY$ is globally contractive,
        \item[(ii)] the morphism $\CZ \to \CZ' \underset{\CY'}{\times} \CY$ is an open embedding,
        \item[(iii)] $\CZ \to \CZ'$ is surjective, and
        \item[(iv)] $\CY \to \CY'$ is smooth.
    \end{enumerate}
\end{defn}

\sssec{} \label{locality of contractiveness}
If we have a commutative diagram as in the above definition, except the top arrow is only contractive instead of globally contractive, then we can still conclude that $\CZ' \to \CY'$ is contractive.

\begin{prop} \cite[Corollary 5.2.3]{DG2}
    If a locally closed substack is contractive, then it is truncative.
\end{prop}

\sssec{}
We finish this subsection by giving a convenient way to show that a given locally closed substack is contractive.

\begin{prop} \label{contraction} \cite[Lemma 5.4.3]{DG2}
    Let $\pi : \CW \to \CS$ be an affine schematic morphism of stacks, such that $\BA^1$ acts on $\CW$ by $\CS$-endomorphisms. Suppose that
    \begin{enumerate}
        \item[(i)] the $\CS$-endomorphism of $\CW$ corresponding to $0 \in \BA^1$ factors as $i \circ \pi$ for some section $i : \CS \hookrightarrow \CW$ of $\pi$ and
        \item[(ii)] the $\BG_m$-action on $\CW$ is trivial over a point (however, this does not imply that the $\BG_m$-action on $\CW$ is trivial over $\CS$).
    \end{enumerate}
    Then $i : \CS \hookrightarrow \CW$ is contractive.
\end{prop}

\ssec{Stratification of $\Bun_G$} \label{stratification}

\sssec{}
In this subsection, we will record some important facts about (but not define) the Harder-Narasimhan stratification of $\Bun_G$, as well as some related combinatorics. The Harder-Narasimhan stratification of $\Bun_G$ is due to Harder and Narasimhan for $G = \on{GL}_n$ and was defined by Schieder in \cite{Sch} for an arbitrary reductive group $G$.

\sssec{}
There are quasicompact locally closed substacks
$$
  \Bun_G^{(\lambda)} \subset \Bun_G
$$
called the \emph{Harder-Narasimhan Strata}, indexed over $\lambda \in \Lambda_G^{+, \BQ}$, the rational dominant coweights of $G$. A definition is given in \cite[Theorem 7.4.3]{DG2}. Given a subset $S \subset \Lambda_G^{+, \BQ}$, we will denote
$$
  \Bun_G^{(S)} := \underset{\lambda \in S}{\bigcup} \Bun_G^{(\lambda)}.
$$

\sssec{}
The $\Bun_G^{(\lambda)}$ form a stratification of $\Bun_G$. In other words, they are disjoint and their union is all of $\Bun_G$. See \cite[Theorem 7.4.3]{DG2} for details.

\sssec{}
The subset $\CA \subset \Lambda_G^{+, \BQ}$ on which $\Bun_G^{(\lambda)}$ is nonempty is discrete. In particular, given $\theta \in \Lambda_G^{+, \BQ}$, the subset
$$
  \{ \lambda \in \CA \mid \lambda \leq \theta \}
$$
is finite, so
$$
  \Bun_G^{(\leq \theta)} := \underset{\lambda \leq \theta}{\bigcup} \Bun_G^{(\lambda)}
$$
is quasicompact.

\sssec{}
The topology of the Harder-Narasimhan Stratification is described in \cite[Corollary 7.4.11]{DG2}. The punchline is that for $\lambda \leq \lambda'$, $\Bun_G^{\lambda'}$ is contained in the closure of $\Bun_G^{\lambda}$. Combined with the discreteness of $\CA$ above, this implies that $\Bun_G^{(\leq \theta)}$ is an open substack of $\Bun_G$.

\sssec{}
Now we shift our focus to $\Lambda_G^{+, \BQ}$ itself. First, let us note that for any standard parabolic $P$ with Levi $M$, we can consider
$$
  \Lambda_{G, P}^{++, \BQ} := \{\lambda \in \Lambda_G^{+, \BQ} \mid \langle \lambda , \check\alpha_i \rangle = 0 \ \forall i \in \Gamma_M, \langle \lambda , \check\alpha_i \rangle > 0 \ \forall i \notin \Gamma_M\}
$$
(Here and in the sequel, $\Gamma_H$ for a reductive group $H$ will denote the set of nodes of the Dynkin diagram of $H$. We will identify it with the set of simple roots of $H$ and with the set of simple coroots of $H$.)

We call a coweight $\lambda \in \Lambda_{G, P}^{++, \BQ}$ \emph{$P$-regular}.

\begin{defn} \cite[Definition 8.2.2]{DG2}
    A subset $S \subset \Lambda_G^{+, \BQ}$ is \emph{$P$-admissible} if 
    \begin{enumerate}
        \item [(1)] there exists some $P$-regular $\mu$ such that 
        $$
        S \subset \mu + \underset{i \in \Gamma_M}{\Sum} \BQ_{\geq 0} \alpha_i,
        $$
        \item [(2)] if $\lambda_1 \in S$ and $\lambda_2 \in \Lambda_G^{+, \BQ}$, then $\lambda_1 - \lambda_2 \in \underset{i \in \Gamma_M}{\Sum} \BQ_{\geq 0} \alpha_i$ implies that $\lambda_2 \in S$, and
        \item [(3)] for all $\lambda \in S$, $\langle \lambda, \check\alpha_i \rangle > 0$ whenever $i \notin \Gamma_M$.
    \end{enumerate}
    Furthermore, we will say that $S$ is \emph{contractively $P$-admissible} if $S$ is $P$-admissible and for all $\lambda \in S$,
    $$
      \langle \lambda , \check\alpha_i \rangle > 2g-2
    $$
    whenever $i \notin \Gamma_M$.
\end{defn}

\ssec{Truncatability of $\Bun_G$} \label{truncatability}

\sssec{}
In this subsection we will finally conclude that $\Dmod(\Bun_G)$ is compactly generated. We will do this by writing $\Bun_G$ as a union of quasicompact co-truncative open substacks, and we will see that the open substacks are co-truncative by observing that their complements are contractive, and therefore truncative.

\begin{prop} \label{cotruncative open substacks} \cite[Theorem 9.1.2]{DG2}
    The open substacks 
    $$
    \Bun_G^{(\leq \theta)} = \underset{\lambda \leq \theta}{\bigcup} \Bun_G^{(\lambda)}
    $$
    are co-truncative whenever
    $$
    \langle \theta, \check\alpha_i \rangle \geq 2g - 2
    $$
    for all simple roots $\check\alpha_i \in \Gamma_G$.
\end{prop}

\begin{prop} \label{contractiveness of Bun_M} \cite[Lemma 11.3.4]{DG2}
Let $S \subset \Lambda_G^{+, \BQ}$ be a contractively $P$-admissible subset. Then we have a commutative diagram
$$
 \begin{tikzcd}
    \Bun_M^{(S)} \arrow[r, "\iota^{(S)}_{P^-}"] \arrow[d, "\iota^{(S)}_P"'] & \Bun_{P^-}^{(S)} \arrow[d, "p^{(S)}_{P^-}"] \\
    \Bun_{P}^{(S)} \arrow[r, "p^{(S)}_{P}"'] & \Bun_G
  \end{tikzcd}
$$
such that
\begin{enumerate}
    \item [(a)] $\Bun_M^{(S)} \to \Bun_P^{(S)} \underset{\Bun_G}{\times} \Bun_{P^-}^{(S)}$ is an open embedding.
    \item [(b)] $\Bun_{M}^{(S)} \to \Bun_P^{(S)}$ is surjective.
    \item [(c)] $\Bun_{P^-}^{(S)} \to \Bun_G$ is smooth.
    \item [(d)] $\Bun_M^{(S)} \to \Bun_{P^-}^{(S)}$ is a section of the affine schematic map $\Bun_{P^-}^{(S)} \to \Bun_M^{(S)}$, and therefore it is a closed embedding.
    \item [(e)] $\Bun_{P}^{(S)} \to \Bun_G$ is an isomorphism onto $\Bun_G^{(S)} \subset \Bun_G$.
\end{enumerate}
\end{prop}

\sssec{} \cite[Proposition 11.3.7]{DG2}
Furthermore, there exists an action of $\BA^1$ on $\Bun_{P^-} \to \Bun_M$ that satisfies the assumptions of Proposition \ref{contraction}, and this action preserves $\Bun_{P^-}^{(S)}$. For the construction of this action, see \cite[11.2]{DG2}. This implies that $\Bun_M^{(S)} \to \Bun_{P^-}^{(S)}$ is contractive, and stability properties of contractiveness imply that $\Bun_{P}^{(S)} = \Bun_{G}^{(S)} \to \Bun_{G}$ is also contractive.

\sssec{}
We want to see that $\Bun_G^{(\leq \theta)}$ is co-truncative in $\Bun_G$ whenever $\theta$ is ``deep enough'' in every direction as in the statement of Proposition \ref{cotruncative open substacks}. In other words, we want to check that $\Bun_G^{(\leq \theta)} \cap \oCY$ is co-truncative in $\oCY$ for all quasicompact open substacks $\oCY$ of $\Bun_G$. However, it actually suffices to check this for a family of quasicompact open substacks that covers $\Bun_G$, and a convenient such family is
$$
  \{ \Bun_G^{(\leq \theta')} \}_{\theta' \geq \theta}.
$$

\sssec{} \label{finite union of admissible}
Now we observe that
$$
  \Bun_G^{(\leq \theta')} \setminus \Bun_G^{(\leq \theta)} = \underset{\lambda \nleq \theta, \lambda \leq \theta'}{\bigcup} \Bun_G^{(\lambda)}
$$
and we write this as a finite union of $\Bun_G^{(S)}$ for contractively $P$-admissible subsets $S \subset \Lambda_G^{+, \BQ}$ (where $P$ is allowed to vary).

\sssec{} \cite[9.3.2]{DG2}
Let 
$$
(\theta, \theta'] = \{ \lambda \in \Lambda_G^{+, \BQ} \mid \lambda \nleq \theta, \lambda \leq \theta'\}.
$$ 
For each $\lambda \in (\theta, \theta']$ we will describe how to produce a parabolic subgroup $P_{\lambda}$ and a contractively $P_{\lambda}$-admissible subset $S_{\lambda} \subset (\theta, \theta']$ containing $\lambda$.

\sssec{}
Let $P_{\lambda}$ be the parabolic subgroup whose Levi $M_{\lambda}$ corresponds to the subset
$$
  \{ i \in \Gamma_G \mid \langle \lambda, \check{\alpha}_i \rangle \leq 2g-2 \}.
$$
Now we take $S_{\lambda}$ to be the contractively $P_{\lambda}$-admissible subset
$$
  \{ \mu \in \Lambda_G^{+, \BQ} \mid \lambda - \mu \in \underset{i \in \Gamma_{M_\lambda}}{\Sum} \BQ_{\geq 0} \alpha_i \}.
$$

\sssec{}
Note that $\CA \cap (\theta, \theta']$ is finite, so
$$
\Bun_G^{(\leq \theta')} \setminus \Bun_G^{\leq \theta} = \underset{\lambda \in \CA \cap (\theta, \theta']}{\bigcup} \Bun_G^{(S_\lambda)}
$$
is a finite union of truncative substacks. Since finite unions of truncative locally closed substacks are truncative, this is sufficient to conclude Proposition \ref{cotruncative open substacks}.

\sssec{}
Clearly the $\Bun_G^{(\leq \theta)}$ cover $\Bun_G$, so Proposition \ref{cotruncative open substacks} implies that 

\begin{prop}
    $\Bun_G$ is truncatable. Since $\Bun_G$ is locally QCA, by Proposition \ref{truncatable}, $\Dmod(\Bun_G)$ is compactly generated.
\end{prop}

\section{The Case $G = \on{SL}_2$} \label{SL2}
In this section, we will examine the case of $G = \on{SL}_2$. In both the unramified and Iwahori ramified versions, most of the conceptual ideas of the proof of compact generation can be seen in this case.

\ssec{Compact Generation of $\Dmod(\Bun_{\on{SL}_2})$} \label{SL2 unramified}

\sssec{}
In this subsection, we will reproduce the discussion from \cite[Section 6]{DG2} that describes the argument that $\Dmod(\Bun_G)$ is compactly generated for $G = \on{SL}_2$.

\sssec{}
The locally QCA algebraic stack $\Bun_G$ represents principal $\on{SL}_2$-bundles on $X$, i.e. rank $2$ vector bundles $\M$ with trivial determinant.

\sssec{}
$\Bun_G$ is an increasing union of the quasicompact open substacks $\Bun_G^{(\leq m)}$. $\Bun_G^{(\leq m)}$ represents $\on{SL}_2$-bundles $\M$ such that every line subbundle of $\M$ has degree at most $m$. 

\sssec{}
To see that the $\Bun_G^{(\leq m)}$ are co-truncative for $m \geq \on{max}(g-1, 0)$, we want to show that $\Bun_G^{(n)} = \Bun_G^{(\leq n)} \setminus \Bun_G^{(\leq n-1)}$ is truncative for $n > \on{max}(g-1, 0)$.

\sssec{}
Now consider the stack $\Bun_B$ that classifies short exact sequences
$$
0 \to \L \to \M \to \L^{-1} \to 0
$$
and the substacks $\Bun_B^n$ that classify such short exact sequences with $\on{deg}(\L) = n$.

\sssec{}
We have a fiber square
\[\begin{tikzcd}
	{\Bun_T^n} && {\Bun_{B}^{-n}} \\
	& {} \\
	{\Bun_B^n} && {\Bun_G^{(\leq n)}}
	\arrow["{i^-}", from=1-1, to=1-3]
	\arrow["\lrcorner"{anchor=center, pos=0.125}, draw=none, from=1-1, to=2-2]
	\arrow["i"', from=1-1, to=3-1]
	\arrow["{p^-}", from=1-3, to=3-3]
	\arrow["p"', from=3-1, to=3-3]
\end{tikzcd}\]
where the stack $\Bun_T$ classifies line bundles $\L$ on $X$, and the substacks $\Bun_T^n$ classify degree $n$ line bundles, and the maps $i$ and $i^-$ sends $\L$ to
$$
0 \to \L \to \L \oplus \L^{-1} \to \L^{-1} \to 0
$$
and
$$
0 \to \L^{-1} \to \L \oplus \L^{-1} \to \L \to 0
$$
respectively.

\sssec{}
The map $p$ is an isomorphism onto the substack $\Bun_G^{(n)} \subset \Bun_G^{(\leq n)}$ essentially because an $\on{SL}_2$-bundle can have at most one line subbundle of degree $n$, since if $\L, \L'$ are two such subbundles, we obtain a map
$$
\L' \to \M \to \L^{-1} \in \on{Hom}(\L', \L^{-1}) = \on{H}^0(X, \L'^{-1} \otimes \L^{-1}) = 0
$$
as $\on{deg}(\L'^{-1} \otimes \L^{-1}) = -2n < 0$ (and vice versa), so $\L = \L'$.

\sssec{}
The map $i$ is a surjection because every $\on{SL}_2$-bundle with a line subbundle of degree $n$ decomposes as a direct sum, since $\on{Ext}^1(\L^{-1}, \L) = \on{H}^1(X, \L^{\otimes 2}) = 0$, as $\on{deg}(\L^{\otimes 2}) = 2n > 2g - 2$.

\sssec{}
The map $p^-$ is smooth because the cofiber of the differential of $p^-$ at a point
$$
0 \to \L^{-1} \to \M \to \L \to 0
$$
is $\on{H}^1(X, \L^{\otimes 2}) = 0$.

\sssec{}
Finally, here is a description of the map $\i^-$. There is a vector bundle $\V$ over $\Bun_T^n$ with fiber over $\L \in \Bun_T^n$ equal to $\on{Ext}^1(\L, \L^{-1}) = \on{H}^1(X, \L^{\otimes -2})$. Now let $\Pic^n$ denote the scheme that classifies degree $n$ line bundles on $X$ together with a trivialization at some fixed point $x \in X$. Now consider the map
$$
s: \Pic^n \to \Bun_T^n
$$
given by forgetting the choice of realization. This realizes $\Pic^n$ as a $\mathbb{G}_m$ torsor over $\Bun_T^n$.
Pulling the vector bundle $\V$ back along $s$ yields a vector bundle over $Pic^n$. Taking the zero section
$$
\Pic^n \to s^* \V
$$
and quotienting by $\mathbb{G}_m$ 
$$
\Pic^n/\mathbb{G}_m \to s^* \V/\mathbb{G}_m
$$
identifies with the map
$$
i^-: \Bun_T^n \to \Bun_B^{-n}
$$
so $i^-$ is contractive.

\ssec{Compact Generation of $\Dmod(\Bun_{\on{SL}_2}^{\on{I}})$} \label{SL2 Iwahori}

\sssec{}
Now we move on to consider the $G = \on{SL}_2$ case in the Iwahori ramified setting.

\sssec{}
The locally QCA algebraic stack $\Bun_G^{\on{I}}$ represents pairs $(\M, L)$ where $\M$ is a principal $\on{SL}_2$-bundle on $X$, and $L$ is a one-dimensional subspace of the two-dimensional vector space $x^*\M$. There is of course a map $\Bun_G^{\on{I}} \to \Bun_G$ given by forgetting $L$.

\sssec{}
The stack $\Bun_G^{\on{I}}$ is an increasing union of the quasicompact open substacks $\Bun_G^{\on{I}, (\leq m)}$, which is the base change of $\Bun_G^{(\leq m)}$ along $\Bun_G^{\on{I}} \to \Bun_G$.

\sssec{}
To see that the $\Bun_G^{\on{I}, (\leq m)}$ are co-truncative for $m \geq \on{max}(g-1, 0)$, we want to show that $\Bun_G^{\on{I}, (n)} = \Bun_G^{\on{I}, (\leq n)} \setminus \Bun_G^{\on{I}, (\leq n-1)}$ is truncative for $n > \on{max}(g-1, 0)$.

\sssec{}
The stack $\Bun_B^{\on{I}}$ classifies pairs
$$
 (0 \to \L \to \M \to \L^{-1} \to 0, L),
$$
and for $n > \on{max}(g-1, 0)$
$$
 \Bun_B^{\on{I}, n} \to \Bun_G^{\on{I}, (n)}
$$
is an isomorphism.

\sssec{}
The stack $\Bun_B^{\on{I}}$ has two natural strata. For a given point of $\Bun_B^{\on{I}}$, we can ask whether or not $L$ is equal to the fiber $x^* \L \subset x^* \M$. The closed strata, denoted by $\prescript{\backslash}{}{\Bun_B^{\on{I}}}$, can be identified with $\Bun_B$, and the open strata we will denote by $\prescript{\circ}{}{\Bun_B^{\on{I}}}$.

\sssec{}
First, let us consider the open strata. There is a fiber square
\[\begin{tikzcd}
	{\Bun_T^n} && {\Bun_{B}^{-n}} \\
	& {} \\
	{\prescript{\circ}{}{\Bun_B^{\on{I}, n}}} && {\Bun_G^{\on{I}, (\leq n)}}
	\arrow["{i^-}", from=1-1, to=1-3]
	\arrow["\lrcorner"{anchor=center, pos=0.125}, draw=none, from=1-1, to=2-2]
	\arrow["i"', from=1-1, to=3-1]
	\arrow["{p^-}", from=1-3, to=3-3]
	\arrow["p"', from=3-1, to=3-3]
\end{tikzcd}\]
where $i$ sends $\L$ to
$$
(0 \to \L \to \L \oplus \L^{-1} \to \L^{-1} \to 0, x^*\L^{-1}) \in \prescript{\circ}{}{\Bun_B^{\on{I}, n}}.
$$

\sssec{}
We claim that $i$ is surjective. Let $(0 \to \L \to \M \to \L^{-1} \to 0, L)$ be a point of $\prescript{\circ}{}{\Bun_B^{\on{I}, n}}$. We already know that $\M$ splits as $\L \oplus \L^{-1}$, so we may consider $L \in x^*\L \oplus x^*\L^{-1}$. Let
$$
v \in x^*\L, w \in x^*\L^{-1}
$$
be a basis of $x^*\M$.
Since $L \neq x^*\L$, there exists some vector space automorphism $$\phi \in \on{Aut}(x^*\M)$$ of the form
$$
\phi(v) = v, \phi(w) = cv + w
$$
that sends $L$ to $x^*\L^{-1}$.

We will now see that $\phi$ can be lifted to an automorphism of $\M$. We have
$$
\on{End}(\M) = \on{Hom}(\L, \L) \oplus \on{Hom}(\L^{-1}, \L) \oplus \on{Hom}(\L, \L^{-1}) \oplus \on{Hom}(\L^{-1}, \L^{-1}).
$$
$\on{Hom}(\L, \L^{-1}) = \on{H}^0(X, \L^{\otimes -2}) = 0$ since $\on{deg}(\L^{\otimes -2}) < 0$. On the other hand, $\on{Hom}(\L^{-1}, \L) = \on{H}^0(X, \L^{\otimes 2})$, and we know that since $\on{deg}(\L^{\otimes 2}) > 2g - 1$, there exists a global section $f$ that does not vanish at $x \in X$. Without loss of generality we may normalize it so that its value at $x$ is $c$. Now the automorphism of $\M$ given by
$$
\Phi = (\on{id}_{\L}, f, 0, \on{id}_{\L^{-1}}) \in \on{Hom}(\L, \L) \oplus \on{Hom}(\L^{-1}, \L) \oplus \on{Hom}(\L, \L^{-1}) \oplus \on{Hom}(\L^{-1}, \L^{-1})
$$
restricts to $\phi$.

Thus $\Phi$ yields an isomorphism
$$
(0 \to \L \to \M \to \L^{-1} \to 0, L) \iso (0 \to \L \to \M \to \L^{-1} \to 0, x^*\L^{-1}),
$$
the latter of which is clearly in the image of $i$.

\sssec{}
The map $p^-$ is smooth, because the cofiber of the differential of $p^-$ at a point 
$$
0 \to \L^{-1} \to \M \to \L \to 0
$$
is $\on{H}^1(X, \L^{\otimes 2}(-x)) = 0$ since $\on{deg}(\L^{\otimes 2}(-x)) \geq 2g - 1$.

\sssec{}
The map $i^-$ is the same map as in the unramified case, and we have already seen that it is contractive.

\sssec{}
There is also a fiber square for the closed strata:
\[\begin{tikzcd}
	{\Bun_T^n} && {\prescript{\circ}{}{\Bun_{B}^{\on{I}, -n}}} \\
	& {} \\
	{\Bun_B^n} && {\Bun_G^{\on{I}, (\leq n)}}
	\arrow["{i^-}", from=1-1, to=1-3]
	\arrow["\lrcorner"{anchor=center, pos=0.125}, draw=none, from=1-1, to=2-2]
	\arrow["i"', from=1-1, to=3-1]
	\arrow["{p^-}", from=1-3, to=3-3]
	\arrow["p"', from=3-1, to=3-3]
\end{tikzcd}\]

\sssec{}
Here $i$ is the same map as in the unramified case, so we know that it is surjective.

\sssec{}
The map $p^-$ can be written as a composition
$$
\prescript{\circ}{}{\Bun_{B}^{\on{I}, -n}} \to \Bun_{B}^{\on{I}, -n} \to \Bun_G^{\on{I}, (\leq n)}.
$$
The first map is an open embedding, and the second map is a base change of the smooth map $\Bun_B^{-n} \to \Bun_G^{(\leq n)}$, so we conclude that $p^-$ is smooth.

\sssec{}
Here is a description of the map $i^-$. There is a vector bundle $\V$ over $\Bun_T^n$ with fiber over $\L \in \Bun_T^n$ equal to $\on{H}^1(X, \L^{\otimes -2}(-x))$. Again, we can take the pullback of $\V$ along $s: \Pic^n \to \Bun_T^n$. This time taking the zero section
$$
\Pic^n \to s^* \V
$$
and quotienting by $\mathbb{G}_m$
$$
\Pic^n/\mathbb{G}_m \to s^* \V/\mathbb{G}_m
$$
identifies with 
$$
i^- : \Bun_T^n \to \prescript{\circ}{}{\Bun_B^{\on{I}, -n}}
$$
so the latter is contractive.

\section{Compact Generation of $\Dmod(\Bun_G^{\on{I}})$} \label{BunGI}
In this section, we will prove our main result, i.e. that $\Dmod(\Bun_G^{\on{I}})$ is compactly generated.

\ssec{Reduction Steps} \label{reduction steps}
\sssec{}
Like in the unramified case, our strategy will be to show that $\Bun_G^{\on{I}}$ is truncatable by writing $\Bun_G^{\on{I}}$ as a union of quasicompact co-truncative open substacks. Since $\Bun_G^{\on{I}}$ is locally QCA, we may then invoke Proposition \ref{truncatable} to conclude that $\Dmod(\Bun_G^{\on{I}})$ is compactly generated.

\sssec{}
From the beginning, we fix a point $x \in X$ (where the Iwahori level structure will live), and this induces a map
$$
\Bun_G = \Maps(X, \on{pt}/G) \to \Maps(\on{pt}, \on{pt}/G) = \on{pt}/G.
$$
Using this map, we define

\begin{defn} The \emph{moduli stack of principal $G$-bundles on $X$ with Iwahori level structure at $x \in X$} is given by the fiber product 
    $$
    \Bun_G^{\on{I}} := \Bun_G \underset{\on{pt}/G}{\times} \on{pt}/B.
    $$
Note that there is a canonical map
    $$
    \Bun_G^{\on{I}} \to \Bun_G
    $$
given by forgetting the Iwahori level structure.
\end{defn}

\sssec{}
Denote for $\lambda \in \Lambda_G^{+, \BQ}$
$$
\Bun_G^{\on{I}, (\lambda)} := \Bun_G^{(\lambda)} \underset{\on{pt}/G}{\times} \on{pt}/B = \Bun_G^{(\lambda)} \underset{\Bun_G}{\times} \Bun_G^{\on{I}}
$$
and for $S \subset \Lambda_G^{+, \BQ}$
$$
\Bun_G^{\on{I}, (S)} :=\Bun_G^{(S)} \underset{\on{pt}/G}{\times} \on{pt}/B = \Bun_G^{(S)} \underset{\Bun_G}{\times} \Bun_G^{\on{I}}.
$$

\sssec{}
Since the map 
$$
\Bun_G^{\on{I}} \to \Bun_G
$$
is finite type, the open substacks
$$
\Bun_G^{\on{I}, (\leq \theta)} := \Bun_G^{(\leq \theta)} \underset{\Bun_G}{\times} \Bun_G^{\on{I}}
$$
are quasicompact and we have
$$
\Bun_G^{\on{I}} = \underset{\theta \in \Lambda_G^{+, \BQ}}{\bigcup} \Bun_G^{\on{I}, (\leq \theta)}.
$$
Therefore, similarly to the non-Iwahori level structure case, it is sufficient to show that

\begin{prop} \label{first reduction}
$\Bun_G^{\on{I}, (\leq \theta)}$ is co-truncative whenever 
$$
   \langle \theta, \check\alpha_i \rangle \geq 2g - 1
$$
for all simple roots $\check\alpha_i \in \Gamma_G$.
\end{prop}

\sssec{}
Note that the bound has changed from $2g - 2$ to $2g - 1$. The reason why will become clear in the sequel. For now, we define

\begin{defn}
    A $P$-admissible subset $S \subset \Lambda_G^{+, \BQ}$ is Iwahori-contractively $P$-admissible if it is $P$-admissible and for all $\lambda \in S$,
    $$
    \langle \lambda, \check\alpha_i \rangle > 2g - 1
    $$
    whenever $i \notin \Gamma_M$.
\end{defn}

\sssec{}
In order to show that $\Bun_G^{\on{I}, (\leq \theta)}$ is co-truncative, we note that for $\theta' \geq \theta$, as in Section \ref{finite union of admissible}, we can write
$$
\Bun_G^{\on{I}, (\leq \theta')} \setminus \Bun_G^{\on{I}, (\leq \theta)}
$$
as a finite union of 
$$
\Bun_G^{\on{I}, (S)}
$$
for Iwahori-contractively $P$-admissible subsets $S \subset \Lambda_G^{+, \BQ}$ (where $P$ is allowed to vary).

\sssec{}
Now we come to the first point of divergence from the proof of the unramified case. We will further stratify the substacks $\Bun_G^{\on{I}, (S)}$ above using \emph{relative position}.

\sssec{}
Let $S$ be Iwahori-contractively $P$-admissible as above. Just as we have defined $\Bun_G^{\on{I}}$, we define
$$
\Bun_P^{\on{I}} := \Bun_P \underset{\on{pt}/G}{\times} \on{pt}/B
$$
and similarly we denote
$$
\Bun_P^{\on{I}, (S)} := \Bun_{P}^{(S)} \underset{\on{pt}/G}{\times} \on{pt}/B.
$$
Observe that we have a map
$$
\Bun_P^{\on{I}, (S)} \to \Bun_G^{\on{I}}
$$
obtained by base change from $\Bun_P^{(S)} \to \Bun_G$. This map induces an isomorphism 
$$
\Bun_P^{\on{I}, (S)} \iso \Bun_G^{\on{I}, (S)}.
$$

\sssec{}
The stack
$$
\on{pt}/P \underset{\on{pt}/G}{\times} \on{pt}/B = P \backslash G/B
$$
is stratified by locally closed substacks
$$
P \backslash PwB /B \subset P \backslash G/B
$$
where $w$ ranges over $W_M \backslash W$. The map
$$
\Bun_P^{\on{I}, (S)} \to \Bun_P^{\on{I}} \to \on{pt}/P \underset{\on{pt}/G}{\times} \on{pt}/B
$$
allows us to define

\begin{defn}
The \emph{piece of $\Bun_P^{\on{I}, (S)}$ having relative position $w$} is given by the fiber product
$$
\prescript{w}{}{\Bun_P^{\on{I}, (S)}} = \Bun_P^{\on{I}, (S)} \underset{P \backslash G/B}{\times} P \backslash PwB /B.
$$
Using the isomorphism $\Bun_P^{\on{I},(S)} \iso \Bun_G^{\on{I}, (S)}$, we also define $\prescript{w}{}{\Bun_G^{\on{I},(S)}}$ as the image of $\prescript{w}{}{\Bun_P^{\on{I}, (S)}}$.
\end{defn}

\sssec{}
Since finite unions of truncative locally closed substacks are truncative we have reduced Proposition \ref{first reduction} to

\begin{prop} \label{second reduction}
The locally closed substack
$$
\prescript{w}{}{\Bun_G^{\on{I},(S)}} \to \Bun_G^{\on{I}}
$$
is contractive, and therefore truncative.
\end{prop}

\ssec{Outline} \label{outline}

\sssec{}
The rest of this section will be dedicated to proving Proposition \ref{second reduction}, i.e. that for a Iwahori-contractively $P$-admissible subset $S$, $\prescript{w}{}{\Bun_G^{\on{I},(S)}} \to \Bun_G^{\on{I}}$ is contractive.

\sssec{}
We will denote $\underline{B} = B \cap M$, and define
$$
\Bun_M^{\on{I}_M} := \Bun_M \underset{\on{pt}/M}{\times} \on{pt}/\underline{B},
$$
and similarly
$$
\Bun_M^{\on{I}_M, (S)} := \Bun_M^{(S)} \underset{\on{pt}/M}{\times} \on{pt}/\underline{B}.
$$

\sssec{}
For $w \in W_M \backslash W$, there is a unique shortest length lift in $W$, which we will abusively also denote by $w$ from now on. For such $w$, the image of the group homomorphism $\underline{B} \to G$ sending $b \mapsto wbw^{-1}$ is contained in $B$, so we obtain a map of stacks
$$
\on{pt}/\underline{B} \overset{w}{\to} \on{pt}/B.
$$

\sssec{}
The diagram
$$
\begin{tikzcd}
	&& {\on{pt}/\underline{B}} \\
	& {\on{pt}/M} && {\on{pt}/B} \\
	{\on{pt}/P} \\
	& {\on{pt}/G}
	\arrow[from=2-4, to=4-2]
	\arrow[from=1-3, to=2-2]
	\arrow[from=2-2, to=3-1]
	\arrow[from=3-1, to=4-2]
	\arrow["w", from=1-3, to=2-4]
\end{tikzcd}
$$
commutes and
$$
\on{pt}/\underline{B} \to \on{pt}/P \underset{\on{pt}/G}{\times} \on{pt}/B
$$
factors through the substack $P \backslash PwB/B \subset P\backslash G/B$. So the diagram
$$
\begin{tikzcd}
	&& {\Bun_M^{\on{I}_M,(S)}} \\
	& {\Bun_M^{(S)}} && {\on{pt}/\underline{B}} \\
	{\Bun_P^{(S)}} && {\on{pt}/M} && {\on{pt}/B} \\
	& {\on{pt}/P} \\
	&& {\on{pt}/G}
	\arrow[from=3-5, to=5-3]
	\arrow[from=2-4, to=3-3]
	\arrow[from=3-3, to=4-2]
	\arrow[from=4-2, to=5-3]
	\arrow["w", from=2-4, to=3-5]
	\arrow[from=3-1, to=4-2]
	\arrow[from=2-2, to=3-1]
	\arrow[from=2-2, to=3-3]
	\arrow[from=1-3, to=2-2]
	\arrow[from=1-3, to=2-4]
\end{tikzcd}
$$
induces a map
$$
\Bun_M^{\on{I}_M,(S)} \to \Bun_P^{\on{I},(S)}
$$
that factors through $\prescript{w}{}{\Bun_P^{\on{I}, (S)}} \subset \Bun_P^{\on{I}, (S)}$.

\sssec{}
Entirely analogously, the diagram
$$
\begin{tikzcd}
	&& {\on{pt}/\underline{B}} \\
	& {\on{pt}/M} && {\on{pt}/B} \\
	{\on{pt}/P^-} \\
	& {\on{pt}/G}
	\arrow[from=2-4, to=4-2]
	\arrow[from=1-3, to=2-2]
	\arrow[from=2-2, to=3-1]
	\arrow[from=3-1, to=4-2]
	\arrow["w", from=1-3, to=2-4]
\end{tikzcd}
$$
commutes and
$$
\on{pt}/\underline{B} \to \on{pt}/P^- \underset{\on{pt}/G}{\times} \on{pt}/B
$$
factors through the substack $P^- \backslash P^-wB/B \subset P^-\backslash G/B$. So the diagram
$$
\begin{tikzcd}
	&& {\Bun_M^{\on{I}_M,(S)}} \\
	& {\Bun_M^{(S)}} && {\on{pt}/\underline{B}} \\
	{\Bun_{P^-}^{(S)}} && {\on{pt}/M} && {\on{pt}/B} \\
	& {\on{pt}/P^-} \\
	&& {\on{pt}/G}
	\arrow[from=3-5, to=5-3]
	\arrow[from=2-4, to=3-3]
	\arrow[from=3-3, to=4-2]
	\arrow[from=4-2, to=5-3]
	\arrow["w", from=2-4, to=3-5]
	\arrow[from=3-1, to=4-2]
	\arrow[from=2-2, to=3-1]
	\arrow[from=2-2, to=3-3]
	\arrow[from=1-3, to=2-2]
	\arrow[from=1-3, to=2-4]
\end{tikzcd}
$$
induces a map
$$
\Bun_M^{\on{I}_M, (S)} \to \Bun_{P^-}^{\on{I},(S)}
$$
that factors through $\prescript{w}{}{\Bun_{P^-}^{\on{I}, (S)}} \subset \Bun_{P^-}^{\on{I}, (S)}$.

\sssec{}
We thus have a commutative diagram
$$
\begin{tikzcd}
	{\Bun_M^{\on{I}_M,(S)}} && {\prescript{w}{}{\Bun_{P^-}^{\on{I},(S)}}} \\
	\\
	{\prescript{w}{}{\Bun_P^{\on{I}, (S)}}} && {\Bun_G^{\on{I}}}
	\arrow[from=1-1, to=3-1]
	\arrow[from=1-1, to=1-3]
	\arrow[from=1-3, to=3-3]
	\arrow[from=3-1, to=3-3]
\end{tikzcd}
$$
and in light of \ref{locality of contractiveness} we have reduced Proposition \ref{second reduction} to showing
\begin{prop} \label{third reduction}
In the diagram above,
    \begin{enumerate}
        \item [(a)] $\Bun_M^{\on{I}_M,(S)} \to \prescript{w}{}{\Bun_P^{\on{I}, (S)}} \underset{\Bun_G^{\on{I}}}{\times} {\prescript{w}{}{\Bun_{P^-}^{\on{I},(S)}}}$ is an open embedding,
        \item [(b)] $\Bun_M^{\on{I}_M,(S)} \to \prescript{w}{}{\Bun_P^{\on{I}, (S)}}$ is surjective,
        \item [(c)] ${\prescript{w}{}{\Bun_{P^-}^{\on{I},(S)}}} \to \Bun_G^{\on{I}}$ is smooth, and
        \item [(d)] $\Bun_M^{\on{I}_M,(S)} \to \prescript{w}{}{\Bun_{P^-}^{\on{I}, (S)}}$ is contractive.
    \end{enumerate}
\end{prop}

\sssec{}
We will prove Part (a) of Proposition \ref{third reduction} in Subsection \ref{open embedding}, Part (b) in Subsection \ref{surjective}, Part (c) in Subsection \ref{smooth}, and Part (d) in Subsection \ref{contractive}. 

\ssec{Global Generation of $\fn(P)_{\F_M}$} \label{global generation subsection}

\sssec{}
In this subsection, we will show that $\fn(P)_{\F_M}$ is globally generated given some conditions on $\F_M \in \Bun_M$. First let us review the notion of global generation.

\begin{defn}
    A vector bundle $\E$ on $X$ is said to be \emph{globally generated} if
    $$
      \on{H}^0(X, \E) \to \E_x
    $$
    is surjective for all $x \in X$.
\end{defn}

\sssec{}
Here are some useful facts for checking that a vector bundle on a curve is globally generated.

\begin{lem} \label{sufficient}
    The vector bundle $\E$ is globally generated if and only if $\on{H}^1(X, \E(-x)) = 0$ for all $x \in X$.
\end{lem}

\begin{lem} \label{first cohomology}
    Let $\E$ have some filtration
    $$
    0 = \E_0 \subset \E_1 \subset ... \subset \E_m = \E
    $$
    whose subquotients $\E_{i}/\E_{i-1}$ are vector bundles, and let 
    $$
    \on{gr}(\E) = \bigoplus \E_{i}/\E_{i-1}
    $$
    be the associated graded of $\E$ with respect to that filtration. If $\on{H}^1(X, \on{gr}(\E)) = 0$, then $\on{H}^1(X, \E) = 0$.
\end{lem}

\sssec{}
The following now follows immediately from the preceding two lemmas.

\begin{lem} \label{assoc graded}
    Let $\E$ be as above. If $\on{H}^1(X, \on{gr}(\E)) = 0$ and $\on{gr}(\E)$ is globally generated, then $\E$ is globally generated.
\end{lem}

\sssec{}
We will spend the rest of the subsection using the above basic facts about global generation to show that $\fn(P)_{\F_M}$ is globally generated.

\begin{lem} \label{numerology}
    Let $\tilde{G}$ be a reductive group, $V$ a finite dimensional $\tilde{G}$-representation on which $\on{Z}_0(\tilde{G})$ acts via $\check\mu$, and $\F_{\tilde{G}} \in \Bun_{\tilde{G}}^{\lambda, ss}$. If
    $$
      \langle \lambda, \check\mu \rangle > 2g-1,
    $$
    then
    $V_{\F_{\tilde{G}}}$ is globally generated.
\end{lem}
\begin{proof}
    We note that by Serre duality, we have 
    $$
    \on{H}^1(X, V_{\F_{\tilde{G}}}(-x)) = \on{H}^0(X, (V^*)_{\F_{\tilde{G}}}(x) \otimes \omega_X)^{*}.
    $$
    By \cite[Lemma 10.3.2]{DG2}, every line subbundle of $(V^*)_{\F_{\tilde{G}}}$ has degree $\leq - \langle \lambda, \check\mu \rangle < -(2g-1)$, so every line subbundle of $(V^*)_{\F_{\tilde{G}}}(x) \otimes \omega_X$ has degree $< 0$. Thus 
    $$
    \on{H}^0(X, (V^*)_{\F_{\tilde{G}}}(x) \otimes \omega_X) = 0
    $$
    and we are done by Lemma \ref{sufficient}.
\end{proof}

\begin{prop} \label{global generation}
    Let $\F_M \in \Bun_M^{(\lambda)}$ be such that $\langle \lambda, \check\alpha_i \rangle > 2g-1$ whenever $i \notin \Gamma_M$. Then the associated vector bundle $\fn(P)_{\F_M}$ has $\on{H}^1$ equal to $0$ and is globally generated.
\end{prop}
\begin{proof}
    As in the proof of \cite[Proposition 10.4.5]{DG2}, let $P_{\lambda}$ be the parabolic of $M$ corresponding to the subset 
    $$
    \{i \in \Gamma_M \mid \langle \lambda, \check\alpha_i \rangle = 0 \} \subset \Gamma_M
    $$
    and let $M_{\lambda}$ be its levi. Since $\F_M \in \Bun_M^{(\lambda)}$, by \cite[Theorem 7.4.3]{DG2} $\F_M$ has a semistable reduction to $P_{\lambda}$ of degree $\lambda$ which we will denote by $\F_{P_\lambda}$. Denote by $\F_{M_\lambda}$ the induced principal $M_{\lambda}$-bundle along $P_\lambda \to M_\lambda$. Then $\fn(P)_{\F_{M_\lambda}}$ is an associated graded of $\fn(P)_{\F_M} = \fn(P)_{\F_{P_\lambda}}$. It is shown in \emph{loc.cit.} that $\on{H}^1(X, \fn(P)_{\F_{M_\lambda}}) = 0$, so by Lemma \ref{first cohomology}, we conclude that $\on{H}^1(X, \fn(P)_{\F_{M}}) = 0$.
    
    For global generation, by Lemma \ref{assoc graded}, it suffices to show that $\fn(P)_{\F_{M_\lambda}}$ is globally generated. Now observe that, as in \cite[Subsection 10.4.1]{DG2}, $\fn(P)$ can be written as a direct sum of $M_{\lambda}$-representations
    $$
      \bigoplus V_{M_{\lambda}, \check\alpha}
    $$
    where the direct sum is over positive roots $\check\alpha$ of $G$ such that there exists some $i \notin \Gamma_M$ with $\on{coeff}_i(\check\alpha) > 0$. $\on{Z}_0(M_{\lambda})$ acts on $V_{M_{\lambda}, \check\alpha}$ via $\check\alpha$, and since
    $$
      \langle \lambda, \check\alpha \rangle \geq \langle \lambda, \check\alpha_i \rangle \cdot \on{coeff}_i(\check\alpha) \geq \langle \lambda, \check\alpha_i \rangle \cdot 1 > 2g-1,
    $$
    Lemma \ref{numerology} allows us to conclude that ${V_{M_{\lambda}, \check\alpha}}_{\F_{M_{\lambda}}}$ is globally generated.
\end{proof}

\begin{cor} \label{global generation surj}
    Let $\F_M$ be as in Proposition \ref{global generation}. Then the map $\on{H}^0(X, U(P)_{\F_M}) \to U(P)_{x^*\F_M} \iso U(P)$ is surjective.
\end{cor}
\begin{proof}
    In characteristic $0$, there is a bijective $M$-equivariant map $\fn(P) \to U(P)$.
\end{proof}

\begin{cor} \label{global generation smooth}
    Let $\F_{P^-} \in \Bun_{P^-}^{(\lambda)}$ be such that $\langle \lambda, \check\alpha_i \rangle > 2g-1$ whenever $i \notin \Gamma_M$. Then $(\fg/\fp^-)_{\F_{P^-}}$ is globally generated.
\end{cor}
\begin{proof}
    Let $\F_M$ be the induced principal $M$-bundle along the map $P^- \to M$. Then $(\fg/\fp^-)_{\F_{M}} = \fn(P)_{\F_M}$ is an associated graded of $(\fg/\fp^-)_{\F_{P^-}}$. Furthermore, $\on{H}^1(X, \fn(P)_{\F_M}) = 0$, so the result follows from Lemma \ref{assoc graded}.
\end{proof}

\ssec{The Map $\Bun_M^{\on{I_M}, (S)} \to \prescript{w}{}{\Bun_P^{\on{I}, (S)}} \underset{\Bun_G^{\on{I}}}{\times} \prescript{w}{}{\Bun_{P^-}^{\on{I}, (S)}}$ is an Open Embedding} \label{open embedding}

\sssec{}
In this subsection, we will prove Part (a) of Proposition \ref{third reduction}, i.e. that 
$$
\Bun_M^{\on{I}, (S)} \to \prescript{w}{}{\Bun_P^{\on{I}, (S)}} \underset{\Bun_G^{\on{I}}}{\times} {\prescript{w}{}{\Bun_{P^-}^{\on{I},(S)}}}
$$
is an open embedding.

\sssec{}
This will immediately follow from Part (a) of Proposition \ref{contractiveness of Bun_M}, which is the fact that 
$$
\Bun_M^{(S)} \to \Bun_P^{(S)} \underset{\Bun_G}{\times} \Bun_{P^-}^{(S)}
$$
is an open embedding, combined with the following

\begin{prop} \label{fiber product}
The map
$$
    \Bun_M^{\on{I}, (S)} \to \Bun_M^{(S)} \underset{\Bun_P^{(S)} \underset{\Bun_G}{\times} \Bun_{P^-}^{(S)}}{\times} (\prescript{w}{}{\Bun_P^{\on{I}, (S)}} \underset{\Bun_G^{\on{I}}}{\times} {\prescript{w}{}{\Bun_{P^-}^{\on{I},(S)}}})
$$
is an isomorphism.
\end{prop}
\begin{proof}
Note that unwinding the left hand side yields
$$
  \Bun_M^{(S)} \underset{\on{pt}/M}{\times} \on{pt}/\underline{B}.
$$
For the right hand side, first let us note that 
$$
\begin{tikzcd}
	& {\Bun_M^{(S)} \underset{\on{pt}/G}{\times} \on{pt}/B} \\
	{\Bun_M^{(S)}} & {} & {\Bun_P^{\on{I}, (S)} \underset{\Bun_G^{\on{I}}}{\times} \Bun_{P^-}^{\on{I}, (S)}} \\
	& {\Bun_P^{(S)} \underset{\Bun_G}{\times} \Bun_{P^-}^{(S)}} & {} & {\on{pt}/B} \\
	&& {\on{pt}/G}
	\arrow[from=1-2, to=2-1]
	\arrow["\lrcorner"{anchor=center, pos=0.125, rotate=-45}, draw=none, from=1-2, to=2-2]
	\arrow[from=1-2, to=2-3]
	\arrow[from=2-1, to=3-2]
	\arrow[from=2-3, to=3-2]
	\arrow["\lrcorner"{anchor=center, pos=0.125, rotate=-45}, draw=none, from=2-3, to=3-3]
	\arrow[from=2-3, to=3-4]
	\arrow[from=3-2, to=4-3]
	\arrow[from=3-4, to=4-3]
\end{tikzcd}
$$
so
\[\begin{tikzcd}
	{\on{RHS}} && {\prescript{w}{}{\Bun_P^{\on{I}, (S)}} \underset{\Bun_G^{\on{I}}}{\times} \prescript{w}{}{\Bun_{P^-}^{\on{I}, (S)}}} && {P \backslash PwB/B \times P^- \backslash P^- wB/B} \\
	& {} && {} \\
	{\Bun_M^{(S)} \underset{\on{pt}/G}{\times} \on{pt}/B} && {\Bun_P^{\on{I}, (S)} \underset{\Bun_G^{\on{I}}}{\times} \Bun_{P^-}^{\on{I}, (S)}} && {P \backslash G/B \times P^- \backslash G/B} \\
	& {} \\
	{\Bun_M^{(S)}} && {\Bun_P^{(S)} \underset{\Bun_G}{\times} \Bun_{P^-}^{(S)}}
	\arrow[from=1-1, to=1-3]
	\arrow["\lrcorner"{anchor=center, pos=0.125}, draw=none, from=1-1, to=2-2]
	\arrow[from=1-1, to=3-1]
	\arrow[from=1-3, to=1-5]
	\arrow["\lrcorner"{anchor=center, pos=0.125, rotate=45}, draw=none, from=1-3, to=2-4]
	\arrow[from=1-3, to=3-3]
	\arrow[from=1-5, to=3-5]
	\arrow[from=3-1, to=3-3]
	\arrow["\lrcorner"{anchor=center, pos=0.125}, draw=none, from=3-1, to=4-2]
	\arrow[from=3-1, to=5-1]
	\arrow[from=3-3, to=3-5]
	\arrow[from=3-3, to=5-3]
	\arrow[from=5-1, to=5-3]
\end{tikzcd}\]
and thus unwinding the right hand side yields
$$
  (\Bun_M^{(S)} \underset{\on{pt}/G}{\times} \on{pt}/B) \underset{P \backslash G/B \times P^- \backslash G/B}{\times} P \backslash PwB/B \times P^- \backslash P^- wB/B.
$$
The map 
$$
   \Bun_M^{(S)} \underset{\on{pt}/M}{\times} \on{pt}/\underline{B}  \to (\Bun_M^{(S)} \underset{\on{pt}/G}{\times} \on{pt}/B) \underset{P \backslash G/B \times P^- \backslash G/B}{\times} P \backslash PwB/B \times P^- \backslash P^- wB/B
$$
is obtained by base-changing the isomorphism
\begin{align}
\on{pt}/\underline{B} &\to M \backslash G/B \underset{P \backslash G/B \times P^- \backslash G/B}{\times} (P \backslash PwB/B \times P^- \backslash P^- wB/B) \\
&= M\backslash (PwB \cap P^- wB)/B \\
&= M\backslash MwB/B
\end{align}
along $\Bun_M^{(S)} \to \on{pt}/M$.
\end{proof}

\ssec{The Map $\Bun_M^{\on{I_M}, (S)} \to \prescript{w}{}{\Bun_P^{\on{I}, (S)}}$ is Surjective} \label{surjective}

\sssec{}
In this subsection, we will prove Part (b) of Proposition \ref{third reduction}, i.e. that for an Iwahori-contractively $P$-admissible subset $S$,
$$
  \Bun_M^{\on{I}_M, (S)} \to \prescript{w}{}{\Bun_P^{\on{I}, (S)}}
$$
is surjective.

\sssec{}
The following diagram will make the argument easier to follow.
$$
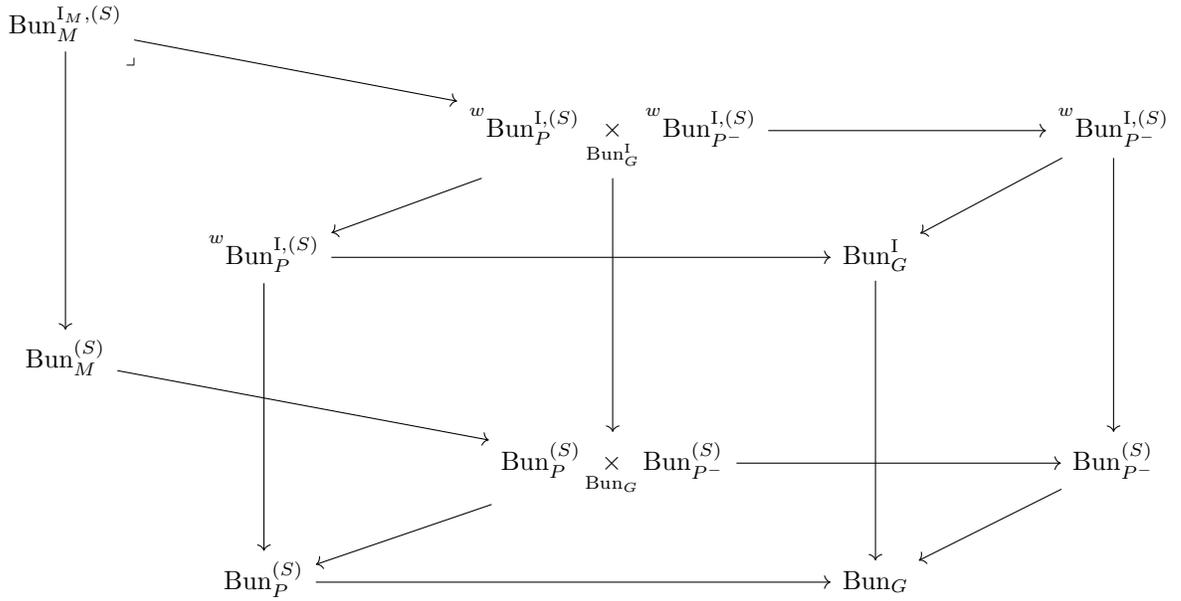
\begin{tikzcd}
	{\Bun_M^{\on{I}_M, (S)}} \\
	& {} && {\prescript{w}{}{\Bun_P^{\on{I},(S)}} \underset{\Bun_G^{\on{I}}}{\times} \prescript{w}{}{\Bun_{P^-}^{\on{I},(S)}}} &&& {\prescript{w}{}{\Bun_{P^-}^{\on{I},(S)}}} \\
	& {\prescript{w}{}{\Bun_P^{\on{I},(S)}}} &&& {\Bun_G^{\on{I}}} \\
	{\Bun_M^{(S)}} \\
	&&& {\Bun_P^{(S)} \underset{\Bun_G}{\times} \Bun_{P^-}^{(S)}} &&& {\Bun_{P^-}^{(S)}} \\
	& {\Bun_P^{(S)}} &&& {\Bun_G}
	\arrow[from=3-2, to=6-2]
	\arrow[from=2-4, to=5-4]
	\arrow[from=3-5, to=6-5]
	\arrow[from=2-7, to=5-7]
	\arrow[from=6-2, to=6-5]
	\arrow[from=5-4, to=5-7]
	\arrow[from=5-7, to=6-5]
	\arrow[from=5-4, to=6-2]
	\arrow[from=2-4, to=3-2]
	\arrow[from=2-7, to=3-5]
	\arrow[from=2-4, to=2-7]
	\arrow[from=3-2, to=3-5]
	\arrow[from=4-1, to=5-4]
	\arrow[from=1-1, to=4-1]
	\arrow[from=1-1, to=2-4]
	\arrow["\lrcorner"{anchor=center, pos=0.125}, draw=none, from=1-1, to=2-2]
\end{tikzcd}
$$

\begin{proof}
    Let
    $$
    y \in \prescript{w}{}{\Bun_P^{\on{I}, (S)}}(k)
    $$
    be a $k$-point and let 
    $$
    \CF_{P} \in \Bun_P^{(S)}(k)
    $$
    be its underlying $P$-bundle. 

    By Part (b) of Proposition \ref{contractiveness of Bun_M}, which is the assertion that
    $$
    \Bun_M^{(S)} \to \Bun_P^{(S)}
    $$
    is surjective, there exists
    $\CF_M \in \Bun_M^{(S)}(k)$
    such that $\CF_{P} = \CF_M \times^M P$.
    Now consider the image of $\CF_M$ under the map
    $$
    \Bun_M^{(S)} \to \Bun_P^{(S)} \underset{\Bun_G}{\times} \Bun_{P^-}^{(S)}.
    $$
    Together with $y$, this gives a $k$-point of the fiber product
    $$
      \prescript{w}{}{\Bun_P^{\on{I}, (S)}} \underset{\Bun_P^{(S)}}{\times} (\Bun_P^{(S)} \underset{\Bun_G}{\times} \Bun_{P^-}^{(S)}) = \prescript{w}{}{\Bun_P^{\on{I}, (S)}} \underset{\Bun_G^{\on{I}}}{\times} \Bun_{P^-}^{\on{I},(S)} \supset \prescript{w}{}{\Bun_P^{\on{I},(S)}} \underset{\Bun_G^{\on{I}}}{\times} \prescript{w}{}{\Bun_{P^-}^{\on{I},(S)}}
    $$

    Now concretely we can think of this $k$-point as the data
    $$
    (\CF_P \in \Bun_P^{(S)}(k), \CF_{P^-} \in \Bun_{P^-}^{(S)}(k), \CF_{B, x} \in \on{pt}/B (k)) 
    $$
    together with isomorphisms
    $$
    (\gamma : \CF_{P} \times^{P} G \iso \CF_{P^-} \times^{P^-} G; \ g: \CF_{B, x} \times^B G \iso x^*\CF_{P} \times^P G)
    $$
    such that $g \in (PwB)_{x^* \CF_{M}} \subset G_{x^* \CF_M}$. Let $u \in U(P)_{x^*\CF_M}$ be such that $ug \in (MwB)_{x^* \CF_M}$ Without loss of generality, we may identify $x^* \gamma \in G_{x^* \CF_M}$ with the identity element. Now we observe that $x^* \gamma \circ ug \in (MwB)_{x^* \CF_M} \subset (P^-wB)_{x^* \CF_M}$.
    
    Corollary \ref{global generation surj}, together with the commutative diagram
    $$
    \begin{tikzcd}
	{\on{H}^0(X, U(P)_{\CF_M})} && {U(P)_{x^*\F_{M}}} \\
	\\
	{\on{H}^0(X, G_{\CF_G})} && G_{x^*\F_{G}}
	\arrow[two heads, from=1-1, to=1-3]
	\arrow[from=1-3, to=3-3]
	\arrow[from=1-1, to=3-1]
	\arrow[from=3-1, to=3-3]
\end{tikzcd}
    $$
    implies that $u$ can be lifted to an automorphism $\alpha \in \on{Aut}(\CF_G) = \on{H}^0(G_{\CF_G})$. Now consider the $k$-point given by the same principal bundles as before, but with the isomorphisms 
    $$
    (\gamma \circ \alpha, g)
    $$
    instead. This gives a $k$-point that lies in
    $$
    \prescript{w}{}{\Bun_P^{\on{I},(S)}} \underset{\Bun_G^{\on{I}}}{\times} \prescript{w}{}{\Bun_{P^-}^{\on{I},(S)}}.
    $$
    Together with the data of $\CF_M$, by Proposition \ref{fiber product} this gives the data of a $k$-point of $\Bun_M^{\on{I}, (S)}$ that is sent to $y$ under
    $$
    \Bun_M^{\on{I}, (S)} \to \prescript{w}{}{\Bun_P^{\on{I}, (S)}},
    $$
    which is the map that we wanted to show was surjective.
\end{proof}

\ssec{The Map $\prescript{w}{}{\Bun_{P^{-}}^{\on{I}, (S)}} \to \Bun_G^{\on{I}}$ is Smooth} \label{smooth}

\sssec{}
In this subsection, we will prove Part (c) of Proposition \ref{third reduction}, i.e. that for an Iwahori-contractively $P$-admissible subset $S$,
$$
\prescript{w}{}{\Bun_{P^{-}}^{\on{I}, (S)}} \to \Bun_G^{\on{I}}
$$
is smooth.

\sssec{}
Since our map is locally of finite presentation between smooth stacks, it suffices to show that for a point 
$$
z \in \prescript{w}{}{\Bun_{P^{-}}^{\on{I}, (S)}}
$$
we have that
$$
H^1(\on{fib}(T_z \prescript{w}{}{\Bun_{P^{-}}^{\on{I}, (S)}} \to T_{f(z)} \Bun_G^{\on{I}})) = 0
$$
but since $S$ is $P$-admissible, $\prescript{w}{}{\Bun_{P^{-}}^{\on{I}, (S)}} \subset \prescript{w}{}{\Bun_{P^{-}}^{\on{I}}}$ is open, so it suffices to show
$$
H^1(\on{fib}(T_z \prescript{w}{}{\Bun_{P^{-}}^{\on{I}}} \to T_{f(z)} \Bun_G^{\on{I}})) = 0.
$$
Here and in the sequel, let us denote
$$
  \on{fib} = \on{fib}(T_z \prescript{w}{}{\Bun_{P^{-}}^{\on{I}}} \to T_{f(z)} \Bun_G^{\on{I}}).
$$

\sssec{}
Note that
$$
\prescript{w}{}{\Bun_{P^-}^{\on{I}}} = \Bun_{P^-} \underset{\on{pt}/P^-}{\times} \on{pt}/(B^w \cap P^-)
$$
thus
\begin{align}
T_z \prescript{w}{}{\Bun_{P^-}^{\on{I}}} &= T_{\F_{P^-}} \Bun_{P^-} \underset{T_{x^*\F_{P^-}}\on{pt}/P^-}{\times} T_{\F_{B^w \cap P^-, x}} \on{pt}/(B^w \cap P^-) \\ &= \Gamma(X, \fp^-[1]_{\F_{P^-}}) \underset{\fp^-[1]_{x^*\F_{P^-}}}{\times} (\fb^w \cap \fp^-)[1]_{\F_{B^w \cap P^-, x}} 
\end{align}
and similarly since
$$
\Bun_G^{\on{I}} = \Bun_G \underset{\on{pt}/G}{\times} \on{pt}/B
$$
we have that 
$$
T_{f(z)}\Bun_G^{\on{I}} = T_{\F_G}\Bun_G \underset{T_{x^*\F_G}\on{pt}/G}{\times} T_{\F_{B,x}} \on{pt}/B
= \Gamma(X, \fg[1]_{\F_G}) \underset{\fg[1]_{x^*\F_G}}{\times} \fb[1]_{\F_{B, x}}
$$
and therefore
$$
\on{fib} = \Gamma(X, (\fg/\fp^-)_{\F_{P^-}}) \underset{(\fg/\fp^-)_{x^*\F_{P^-}}}{\times} (\fb/(\fb^w \cap \fp^-))_{\F_{B^w \cap P^- , x}}
$$
which yields a long exact sequence
$$
\on{H}^0(X, (\fg/\fp^-)_{\F_{P^-}}) \oplus (\fb/(\fb^w \cap \fp^-))_{\F_{B^w \cap P^- , x}} \to (\fg/\fp^-)_{x^*\F_{P^-}} \to H^1(\on{fib}) \to \on{H}^1(X, (\fg/\fp^-)_{\F_{P^-}}).
$$
Now the result follows from the observation that the last term is zero and the first arrow is surjective by Corollary \ref{global generation smooth}.

\ssec{The Map $\Bun_M^{\on{I_M}, (S)} \to \prescript{w}{}{\Bun_{P^{-}}^{\on{I}, (S)}}$ is Contractive} \label{contractive}

\sssec{}
In this subsection, we will prove Part (d) of Proposition \ref{third reduction}, i.e. that 
$$
\Bun_{M}^{\on{I}_M, (S)} \to \prescript{w}{}{\Bun_{P^-}^{\on{I}, (S)}}
$$
is contractive.

\sssec{}
Our strategy is to make use of Proposition \ref{contraction}. The first thing we need to do is to write down a map $\pi : \prescript{w}{}{\Bun_{P^-}^{\on{I}, (S)}} \to \Bun_{M}^{\on{I}_M, (S)}$ that $i$ is a section of.

\sssec{}
Note that we have a diagram

\[\begin{tikzcd}
	&& {\prescript{w}{}{\Bun_{P^-}^{\on{I}, (S)}}} \\
	& {\Bun_{P^-}^{\on{I}, (S)}} & {} & {P^-\backslash P^-wB/B} \\
	{\Bun_{P^-}^{(S)}} & {} & {P^- \backslash G/B} \\
	& {\on{pt}/P^-} & {} & {\on{pt}/B} \\
	&& {\on{pt}/G}
	\arrow[from=2-2, to=3-1]
	\arrow[from=3-1, to=4-2]
	\arrow[from=4-2, to=5-3]
	\arrow[from=3-3, to=4-2]
	\arrow[from=2-2, to=3-3]
	\arrow[from=2-4, to=3-3]
	\arrow[from=4-4, to=5-3]
	\arrow[from=3-3, to=4-4]
	\arrow[from=1-3, to=2-2]
	\arrow[from=1-3, to=2-4]
	\arrow["\lrcorner"{anchor=center, pos=0.125, rotate=-45}, draw=none, from=1-3, to=2-3]
	\arrow["\lrcorner"{anchor=center, pos=0.125, rotate=-45}, draw=none, from=2-2, to=3-2]
	\arrow["\lrcorner"{anchor=center, pos=0.125, rotate=-45}, draw=none, from=3-3, to=4-3]
\end{tikzcd}\]

and $P^- \backslash P^- w B/B$ can be identified with $\on{pt}/(B^w \cap P^-)$, so we have an identification
$$
\prescript{w}{}{\Bun_{P^-}^{\on{I}, (S)}} = \Bun_{P^-}^{(S)} \underset{\on{pt}/P^-}{\times} \on{pt}/(B^w \cap P^-).
$$

\sssec{}
Under this identification, the map $i : \Bun_{M}^{\on{I}_M} \to \prescript{w}{}{\Bun_{P^-}^{\on{I}, (S)}}$ is the one induced by the diagram

\[\begin{tikzcd}
	& {\Bun_M^{\on{I}_M, (S)}} \\
	{\Bun_M^{(S)}} & {} & {\on{pt}/\underline{B}} \\
	& {\on{pt}/M} \\
	{\Bun_{P^-}^{(S)}} && {\on{pt}/(B^w \cap P^-)} \\
	& {\on{pt}/P^-}
	\arrow[from=2-1, to=3-2]
	\arrow[from=2-3, to=3-2]
	\arrow[from=1-2, to=2-1]
	\arrow[from=1-2, to=2-3]
	\arrow["\lrcorner"{anchor=center, pos=0.125, rotate=-45}, draw=none, from=1-2, to=2-2]
	\arrow[from=2-3, to=4-3]
	\arrow[from=4-3, to=5-2]
	\arrow[from=3-2, to=5-2]
	\arrow[from=2-1, to=4-1]
	\arrow[from=4-1, to=5-2]
\end{tikzcd}\]

and now it is clear that $i$ is a section of the map $\pi : \prescript{w}{}{\Bun_{P^-}^{\on{I}, (S)}} \to \Bun_{M}^{\on{I}_M, (S)}$ induced by the diagram

\[\begin{tikzcd}
	& {\prescript{w}{}{\Bun_{P^-}^{\on{I}, (S)}}} \\
	{\Bun_{P^-}^{(S)}} & {} & {\on{pt}/(B^w \cap P^-)} \\
	& {\on{pt}/P^-} \\
	{\Bun_M^{(S)}} && {\on{pt}/\underline{B}} \\
	& {\on{pt}/M}
	\arrow[from=2-3, to=3-2]
	\arrow[from=2-1, to=3-2]
	\arrow[from=4-1, to=5-2]
	\arrow[from=4-3, to=5-2]
	\arrow[from=3-2, to=5-2]
	\arrow[from=2-1, to=4-1]
	\arrow[from=2-3, to=4-3]
	\arrow[from=1-2, to=2-1]
	\arrow[from=1-2, to=2-3]
	\arrow["\lrcorner"{anchor=center, pos=0.125, rotate=-45}, draw=none, from=1-2, to=2-2]
\end{tikzcd}\]

\sssec{}
Now we will give an $\mathbb{A}^1$-action on $\prescript{w}{}{\Bun_{P^-}^{\on{I}, (S)}}$ over $\Bun_{M}^{\on{I}_M, (S)}$ such that $0$ acts as $i \circ \pi$ and the induced $\mathbb{G}_m$-action is trivial over a point.

\sssec{}
As in \cite[Subsection 11.2]{DG2}, let $\mu : \mathbb{G}_m \to Z(M) \subset T$ satisfy $\langle \mu, \check\alpha_i \rangle > 0$ for $i \notin \Gamma_M$. Then the $\mathbb{G}_m$-action on $P^-$ over $M$ defined by
$$
\rho_t(p) = \mu(t)^{-1} p \mu(t)
$$
extends to an $\mathbb{A}^1$-action on $P^-$ over $M$ such that $\rho_0 \in End(P^-)$ corresponds to the composition $P^- \to M \to P^-$. Furthermore, this action sends $B^w \cap P^- \subset P^-$ to itself, so we obtain an $\mathbb{A}^1$-action on $B^w \cap P^-$ over $\underline{B}^w \iso \underline{B}$.

This gives rise to an $\mathbb{A}^1$-action on $\on{pt}/P^-$ over $\on{pt}/M$, as well as a compatible $\mathbb{A}^1$-action on $\on{pt}/(B^w \cap P^-)$ over $\on{pt}/\underline{B}$. Applying $\on{Maps}(X, -)$ to the $\mathbb{A}^1$-action on $\on{pt}/P^-$ over $\on{pt}/M$ yields a (tautologically compatible) $\mathbb{A}^1$-action on $\Bun_{P^-}^{(S)}$ over $\Bun_M^{(S)}$.

Now the map $\prescript{w}{}{\Bun_{P^-}^{\on{I}, (S)}} \to \Bun_M^{\on{I}_M, (S)}$ is of the form
$$
\Bun_{P^-}^{(S)} \underset{\on{pt}/P^-}{\times} \on{pt}/(B^w \cap P^-) \to \Bun_M^{(S)} \underset{\on{pt}/M}{\times} \on{pt}/\underline{B}
$$
and the fiber product of the $\mathbb{A}^1$-actions satisfies the hypotheses of Proposition \ref{contraction}, as desired.

\ssec{Iwahori Level Structure at Multiple Points} \label{multiple points}

\sssec{}
Throughout the paper, we have considered $\Bun_G^{\on{I}}$, the moduli stack of principal $G$-bundles on a curve $X$ with Iwahori level structure at a single point $x \in X$. However, we could have also considered Iwahori level structure at $k$ distinct points $x_1, ..., x_k \in X$. In this subsection, we will discuss what changes need to be made to the proof in order to show that the category of D-modules on the analogous moduli stack is compactly generated.

\begin{defn}
    Let $x_1, ..., x_k \in X$ be $k$ distinct points. Then we define
    $$
    \Bun_G^{(\on{I}; x_1, ..., x_k)} = \Bun_G \underset{(\on{pt}/G)^k}{\times} (\on{pt}/B)^k
    $$
    where the map is
    $$
    (x_1^*, ..., x_k^*) : \Bun_G \to (\on{pt}/G)^k.
    $$
\end{defn}

\begin{prop}
    $\Dmod(\Bun_G^{(\on{I}; x_1, ..., x_k)})$ is compactly generated.
\end{prop}

\sssec{}
It turns out that conceptually, the proof of Proposition $\ref{multiple points}$ is identical to the proof of the single point case; the only difference is in the numerology. Just as the correct bound when going from the Iwahori ramification at no points case (i.e. \cite{DG2}) to the single point case increased from $2g - 2$ to $2g - 1$, the correct bound in the $k$ points case will be $2g - 2 + k$.

\sssec{}
In other words, the proof consists of showing that the quasicompact open substacks
$$
\Bun_G^{(\on{I}; x_1, ..., x_k), (\leq \theta)}
$$
are co-truncative for $\langle \theta, \check{\alpha}_i \rangle \geq 2g - 2 + k$. This is done by writing
$$
\Bun_G^{(\on{I}; x_1, ..., x_k), (\leq \theta')} \setminus \Bun_G^{(\on{I}; x_1, ..., x_k), (\leq \theta)}
$$
as a finite union of $\Bun_G^{(\on{I}; x_1, ..., x_k), (S)}$ for $k$-Iwahori contractively $P$-admissible subsets $S$ (where $P$ is allowed to vary), then stratifying further by
$$
\prescript{(w_1, ..., w_k)}{}{\Bun_G^{(\on{I}; x_1, ..., x_k), (S)}} = \Bun_G^{(\on{I}; x_1, ..., x_k), (S)} \underset{(P\backslash G/B)^k}{\times} (P\backslash Pw_1 B/B \times ... \times P\backslash Pw_k B/B).
$$

\sssec{}
The reason for this boils down to the fact that instead of global generation of $\fn(P)_{\F_M}$, we want to have $k$-global generation of $\fn(P)_{\F_M}$; i.e. we want 
$$
\on{H}^0(X, \fn(P)_{\F_M}) \to \bigoplus_{j=1}^k (\fn(P)_{\F_M})_{x_j}
$$
to be surjective for any $k$ points of $X$.


\begin{thebibliography}{99}

\bibitem[Gai]{IndCoh} D.~Gaitsgory, Ind-coherent Sheaves. arXiv:1105.4857.

\bibitem[DG13]{DG1} V.~Drinfeld and D.~Gaitsgory. On Some Finiteness Questions for algebraic stacks.
{\it Geom. Funct. Anal.} {\bf 23}, no. 1, 149--294, 2013.

\bibitem[DG15]{DG2} V.~Drinfeld and D.~Gaitsgory. Compact Generation of the Category of D-modules on the Stack of G-Bundles on a Curve.
{\it Camb. J. Math.} {\bf 3}, no. 1--2, 19--125, 2015.

\bibitem[AG09]{AG} S.~Arkhipov and D.~Gaitsgory. Localization and the Long Intertwining Operator for Representations of Affine Kac-Moody Algebras. Available at \url{https://people.mpim-bonn.mpg.de/gaitsgde/GL/Arkh.pdf}.

\bibitem[Sch15]{Sch} S.~Schieder. The Harder-Narasimhan Stratification of the Moduli Stack of G-bundles via Drinfeld's Compactifications. {\it Sel. Math. New Ser.} {\bf 21}, 763--831, 2015.

\bibitem[GR24a]{GLC1} D.~Gaitsgory and S.~Raskin. Proof of the Geometric Langlands Conjecture I: Construction of the Functor. arXiv:2405.03599, 2024.

\bibitem[ABC+24a]{GLC2} D.~Arinkin, D.~Beraldo, L.~Chen, J.~Faergeman, D.~Gaitsgory, K.~Lin, S.~Raskin, N.~Rozenblyum. Proof of the Geometric Langlands Conjecture II: Kac-Moody Localization and the FLE. arXiv:2405.03648, 2024.

\bibitem[CCF+24]{GLC3} J.~Campbell, L.~Chen, J.~Faergeman, D.~Gaitsgory, K.~Lin, S.~Raskin, N.~Rozenblyum. Proof of the Geometric Langlands Conjecture III: Compatibility with Parabolic Induction. arXiv:2409.07051, 2024.

\bibitem[ABC+24b]{GLC4} D.~Arinkin, D.~Beraldo, L.~Chen, J.~Faergeman, D.~Gaitsgory, K.~Lin, S.~Raskin, N.~Rozenblyum. Proof of the Geometric Langlands Conjecture IV: Ambidexterity. arXiv:2409.08670, 2024.

\bibitem[GR24b]{GLC5} D.~Gaitsgory and S.~Raskin. Proof of the Geometric Langlands Conjecture V: The Multiplicity One Theorem. arXiv:2409.09856, 2024.


\end{thebibliography}
\end{document}